\documentclass[a4paper]{amsart}

\usepackage{color,graphicx,enumerate,wrapfig,amssymb,mathtools}
\usepackage{epsfig}
\mathtoolsset{showonlyrefs}
\usepackage{hyperref}

\usepackage{eucal}

\usepackage{geometry}
\newcommand{\R}{{\mathbb R}}

\newcommand{\Hf}{{\mathbb{H}}}

\newcommand{\cS}{{\mathcal S}}
\newcommand{\cK}{{\mathcal K}}

\newcommand{\cB}{{\mathcal B}}
\newcommand{\cC}{{\mathcal C}}

\newcommand{\cR}{{\mathcal R}}

\newcommand{\cT}{{\mathcal T}}

\newcommand{\uu}{\mathbf{u}}
\newcommand{\vv}{\mathbf{v}}
\newcommand{\w}{\mathbf{w}}

\newcommand{\h}{\mathbf{h}}
\newcommand{\g}{\mathbf{ g}}

\newcommand{\e}{\varepsilon}
\newcommand{\al}{\alpha}
\newcommand{\de}{\delta}
\newcommand{\ka}{\kappa}
\newcommand{\la}{\lambda}
\newcommand{\si}{\sigma}
\newcommand{\vp}{\varphi}

\newcommand{\dist}{\operatorname{dist}}

\newcommand{\loc}{\operatorname{loc}}

\newcommand{\D}{\nabla}
\newcommand{\p}{\partial}

\newcommand{\mean}[1]{\langle{#1}\rangle}
\newcommand{\lmean}[1]{\left\langle #1\right\rangle}

\usepackage{amsmath}
\usepackage{amsfonts}
\usepackage{caption}
\usepackage{subcaption}
\usepackage{esint}

\theoremstyle{plain}

\numberwithin{equation}{section}

\newtheorem{thm}{Theorem}[section]
\newtheorem{cor}[thm]{Corollary}
\newtheorem{lem}[thm]{Lemma}
\newtheorem{prop}[thm]{Proposition}
\theoremstyle{definition}
\newtheorem{defn}[thm]{Definition}
\theoremstyle{remark}
\newtheorem{rem}[thm]{Remark}
\numberwithin{equation}{section}

\newcommand{\be}{\begin{equation}}
\newcommand{\ee}{\end{equation}}

\newcommand{\eps}{\varepsilon}

\newcommand{\comment}[1]{}

\begin{document}

\author{Daniela De Silva}
\address[Daniela De Silva]
{Department of Mathematics,  Barnard College, Columbia University,
 New York, NY,
10027} 
\email[Daniela De Silva]{desilva@math.columbia.edu}

\author{Seongmin Jeon}
\address[Seongmin Jeon]
{Department of Mathematics Education, Hanyang University,  222 Wangsimni-ro, Seongdong-gu, Seoul 04763, Republic of Korea} 
\email[Seongmin Jeon]{seongminjeon@hanyang.ac.kr}

\author{Henrik Shahgholian}
\address[Henrik Shahgholian]
{Department of Mathematics, KTH Royal Institute of Technology, 100 44 Stockholm, Sweden
} 
\email[Henrik Shahgholian]{henriksh@kth.se}


\title[The free boundary for a superlinear system]
{The free boundary for a superlinear system }



\begin{abstract}

In this paper, we study superlinear systems that give rise to free boundaries. Such systems appear for example from the minimization of the energy functional
$$
\int_{\Omega}\left(|\D\uu|^2+\frac2p|\uu|^p\right),\quad 0<p<1,
$$
but solutions can be also understood in an ad hoc viscosity way.

First, we prove the optimal regularity of minimizers using a variational approach. Then, we apply a linearization technique to establish the $C^{1,\alpha}$-regularity of the ``flat'' part of the free boundary via a viscosity method. Finally, for minimizing free boundaries, we extend this result to analyticity.

\end{abstract}

\maketitle 

\tableofcontents

\section{Introduction and Background}

\subsection{Energy minimizers}

For a bounded open set $\Omega$ in $\R^n$ $(n\ge2)$ and a function $g:\partial \Omega\to\R$, the classical one-phase Alt--Phillips equation concerns non-negative minimizers of
$$
\int_\Omega\left(|\D u|^2+\frac2pu^p
\right),\quad 0<p<2,
$$ 
over $\cK=\{u\in W^{1,2}(\Omega)\,:\, u\geq 0, \  u=g\,\text{ on }\,\partial \Omega\}$. Local  minimizers satisfy the Alt--Phillips equation
\begin{align*}
        \Delta u=u^{p-1} \quad \text{in $\{u>0\}$},\,\,\,
         \qquad |\nabla u|= 0 \quad\text{on $\p \{u>0\} \cap \Omega$.}
\end{align*}
When $p=1$, this is  the classical obstacle problem
$$
\Delta u=\chi_{\{u>0\}},\,\,\, u\ge0\quad\text{in }\Omega.
$$
The regularity of minimizers and their free boundaries $\partial{u>0} \cap \Omega$ was studied in \cite{Phi83}, \cite{AltPhi84}, and several other papers.
The Alt--Phillips problem has since been extended to various settings, including parabolic equations, fully nonlinear operators, nonlocal versions, systems, and many others.

In this paper, we are interested in the following system version. Given $\g:\partial \Omega\to\R^m$, $m\ge1$, we consider minimizers of the energy functional
\begin{align}
    \label{eq:ftnal}
    E(\uu):=\int_\Omega\left(|\D\uu|^2+\frac2p|\uu|^p\right)
\end{align}
among all functions $\uu\in W^{1,2}(\Omega;\R^m)$ with $\uu=\g$ on $\partial \Omega$. Minimizers $\uu$ solve the system
\begin{align}
    \label{eq:system}
    \Delta\uu=|\uu|^{p-2}\uu\chi_{\{|\uu|>0\}}\quad\text{in }\Omega.
\end{align}
When $1\le p<2$, the $W^{2,q}$-regularity of solutions to \eqref{eq:system} (for any $q<\infty$) and the $C^{1,\alpha}$-regularity of the ``regular'' part of the free boundary $$\Gamma(\uu)=\partial\{x\,:\, |\uu(x)|>0\}\cap \Omega$$ were established in \cite{AndShaUraWei15} and \cite{FotShaWei21}. Additionally, for such range of $p$, the higher regularity of the free boundary was proved in \cite{FotKoc24}. 

However, for the remaining case $0<p<1$, regularity results have been an open question for over a decade since the study of the case $p=1$ in \cite{AndShaUraWei15}. The objective of this paper is to fill this gap by establishing similar regularity results to those in \cite{AndShaUraWei15, FotShaWei21, FotKoc24} for the range $0<p<1$.

\subsection{Notation.} Before providing the main results and the structure of the paper, we collect here some notation that we use throughout the paper. 

We denote by $\{e_i\}_{i=1,\ldots,n}$,$\{f^i\}_{i=1,\ldots,m}$ canonical basis in $\R^n$ and $\R^m$ respectively. Unit directions in $\R^n$ and $\R^m$ will be typically denoted by $e$ and $f$. 
The Euclidean norm in either space is denoted by $|\cdot|,$ while the dot product is denoted by $\langle \cdot, \cdot \rangle.$
We indicate the points in $\R^n$ by $x=(x',x_n)$, where $x'=(x_1,\cdots,x_{n-1})\in \R^{n-1}$.
For $x\in\R^n$ and $r>0$, we let
\begin{alignat*}{2}
  B_r(x)&:=\{y\in \R^n:|x-y|<r\}&\quad&\text{ball in $\R^n$,}\\
  B^{\pm}_r(x')&:=B_r(x',0)\cap \{\pm x_n>0\}&& \text{half-ball},\\
  B'_r(x')&:=B_r(x',0)\cap \{x_n=0\} &&\text{thin ball.}
\end{alignat*}
When the center is the origin, we write $B_r=B_r(0)$, $B_r^\pm=B_r^\pm(0)$, and  $B_r'=B'_r(0)$.
For a domain $\Omega$, we indicate the integral mean value of $\uu$ by $$
\overline{\uu}_\Omega:=\fint_\Omega\uu=\frac1{|\Omega|}\int_\Omega\uu.
$$
In particular, when $\Omega=B_r(x_0)$, we simply write
 $$
\overline{\uu}_{x_0,r}:=\overline{\uu}_{B_r(x_0)}.
$$
 We use the notations $\Omega(\uu)$ and $B_1(\uu)$ to denote the positivity set of $|\uu|$ inside $\Omega$ and $B_1$, respectively.

We fix
$$
0<p<1
$$
and set
$$u_0(t)= c_p (t^+)^\ka, \quad \ka:=\frac2{2-p}\in (1,2), \quad c_p:= \left[\ka(\ka-1)\right]^{\frac{1}{p-2}},$$
that is $u_0$ solves the equation
\begin{equation}\label{ODE} u_0''(t) = u_0^{p-1}(t) \quad \text{in $(0, +\infty)$}.\end{equation}

Finally, constants depending only on $n,m$ and $p$ will be called universal.

\subsection{Main results}
In this section, we state our main results. 
Our first result concerns the optimal regularity of minimizers.

\begin{thm}\label{thm:grad-u-holder}
Let $\uu\in W^{1,2}(B_1;\R^m)$ be a minimizer to $E(\cdot)$ in $B_1$. Then $\uu\in C^{1,\ka-1}(B_1)$. Moreover, for any $K\Subset B_1$, 
$$
\|\uu\|_{C^{1,\ka-1}(K)}\le C(n,p,K)\left(\|\uu\|_{W^{1,2}(B_1)}+1\right).
$$
\end{thm} 

We note that for $1 \le p < 2$, applying standard elliptic theory and a bootstrapping argument yields that the solution belongs to $W^{2,q}_{\loc}$ for any $q \in [1, \infty)$.
When $p = 1$, unlike in the classical scalar obstacle problem, it remains an open problem\footnote{It is noteworthy that $\uu$ exhibits quadratic growth from points where $\nabla \uu$ vanishes; see \cite{AndShaUraWei15}.} whether solutions belong to $W^{2,\infty}_{\loc}$.
However, this approach is not applicable when $0 < p < 1$. Instead, inspired by the strategy in \cite{DeSJeoSha22, DeSJeoSha23} for almost minimizers in the range $1 \le p < 2$, we adopt a variational approach to establish the optimal regularity.

\medskip

To state our result on the regularity of the free boundary $\Gamma(\uu)$, we need to define  a slightly unorthodox flatness concept.

\begin{defn}\label{def:flat}(Flatness)
We say $\uu$ is $\bar \eps $-flat if      for some unit directions $e \in \R^n, f\in \R^m$
\be\label{flat} |\uu(x) - u_0(\langle x, e \rangle)f| \leq \bar \eps \quad \text{in $B_1$}
\ee
and
\be\label{nond}
 |\uu| \equiv 0 \quad \text{in $B_1 \cap \{\langle x , e \rangle  < - \bar \eps\}$}.
 \ee
 
\end{defn}


\begin{thm}\label{thm:main}Let $\uu$ be a minimizer to $E(\cdot)$ in $B_1$. There exists a constant $\bar \eps >0$, depending only on $n$, $m$ and $p$, such that if $\uu$ is $\bar \eps$ flat in $B_1$,  
then $\Gamma(\uu)$ is analytic in $B_{1/2}.$
\end{thm}

For $1\le p<2$,  the authors of \cite{AndShaUraWei15, FotShaWei21} obtained the $C^{1,\alpha}$-regularity of the ``regular'' set using an epiperimetric inequality approach. Later, \cite{FotKoc24} improved this result from $C^{1,\alpha}$ to analyticity by applying the partial hodograph-Legendre transformation. Although the conditions on the regular free boundary points in \cite{AndShaUraWei15, FotShaWei21} are formulated differently from our flatness assumption in Theorem~\ref{thm:main}, both essentially consider asymptotically one-phase points, where the blow-ups become half-space solutions (see \eqref{eq:half-space-sol}). 
However, in the case $0 < p < 1$, proving the epiperimetric inequality involves significant technical difficulties. To overcome this, we take a different approach - we use a linearization technique, first applied to free boundary problems in \cite{DeS11}, to show that flat free boundaries are $C^{1,\alpha}$. For this purpose, we need to introduce an ad hoc notion of a viscosity solution to our system, as described in the next subsection.

Our final goal is to prove that the flat free boundary is analytic. In the case of scalar equations, \cite{ResRos24} recently showed that once the free boundary is $C^{1,\alpha}$, it is actually $C^\infty$. Their strategy involved analyzing a linearized operator for the derivatives of the solution, obtaining fine regularity estimates, and then applying a bootstrapping argument to improve the regularity from $C^{1,\alpha}$ to $C^\infty$. However, their method relied on a comparison principle for the linearized equation, which does not seem to be available for our system. Instead, we follow the approach in \cite{FotKoc24}, where the authors proved the analyticity of the regular set for the system \eqref{eq:system} in the range $1\le p<2$.

\subsection{Viscosity solutions}\label{sec:visco}

Parallel to the minimizers, one can introduce the slightly unorthodox concept of viscosity solutions for the superlinear system associated to $E(\cdot)$. 
The notion of a viscosity solution to a free boundary problem is well understood and amply utilized in the scalar setting when the existence of such solutions can be achieved via Perron's method. In the case of systems, despite the lack of a maximum principle, an ad hoc viscosity notion turns out to still be effective in the ``flat" regime. In this perturbative case, one roughly reduces to the scalar setting, as one component of a solution is a supersolution of a scalar problem and the modulus of a solution is a subsolution of a scalar problem. Thus, since this notion is satisfied by minimizers, it is well suited for our purpose of analyzing the question of regularity of the flat part of the minimizing free boundary.

\begin{defn}\label{solutionnew} For a bounded domain $\Omega\subset\R^n$, we say that $\uu \in C(\Omega; \R^m)$ is a viscosity solution to 
\begin{equation}\label{VOP} \Delta \uu = |\uu| ^{p-2} \uu \quad \text{in $\Omega$},\end{equation} if
\be\label{int1}\Delta u^i=|\uu|^{p-2}u^i \quad \text{in $\Omega(\uu):= \Omega \cap \{|\uu|>0\}$}, \quad \forall i=1,\ldots,m;\ee
and the free boundary condition on 
$$\Gamma(\uu):= \Omega \cap \p \Omega(\uu)$$
is satisfied in the following sense: given $x_0 \in \Gamma(\uu)$ and any  $f \in \R^m$,
$\langle \uu, f\rangle$ cannot be touched by below at $x_0$ by a test function $\varphi \in C^{1}$ which satisfies $$\varphi(x_0)=0, \quad |\nabla \varphi|(x_0) \neq 0$$ and in $\{\varphi \neq 0\},$
$$\varphi \in C^2, \quad \Delta \varphi > \varphi^{p-1} \chi_{\{\varphi>0\}}.$$
\end{defn}
We remark that, by interior elliptic regularity for the equations \eqref{int1}, we deduce that $\uu \in C^\infty(\Omega(\uu)).$

Also, \eqref{int1} implies that $u_f:= \langle\uu, f\rangle$ satisfies \begin{equation}\label{super}
\Delta u_f \leq u_f^{p-1}\chi_{\{u_f>0\}} \quad \text{in $\Omega(\uu)$}\end{equation} for any unit vector $f \in \R^m$. This, combined with
the free boundary condition, can be reformulated as 
\begin{equation}\label{super1}
\Delta u_f \leq u_f^{p-1}\chi_{\{u_f>0\}} \quad \text{in $\Omega(\uu) \cup \Gamma(\uu)$,}\end{equation} where the condition on the free boundary is understood in the viscosity sense.

\begin{rem}\label{notouch} Notice that a function $\varphi$ as in the definition above, cannot touch $u_f$ by below in $\Omega(\uu)$. This is obvious on $\{\varphi \neq 0\}$. If $\varphi$ touches $u_f$ at $x_0 \in \{\varphi =0\}$, then since $u_f$ is smooth near $x_0$, $|\nabla u_f(x_0)|=|\nabla \varphi(x_0)|$ and we contradict Hopf's lemma applied to $u_f-\varphi$ in the set $\{u_f<0\}.$
\end{rem}



\begin{rem} We remark that
\be\label{sub2} \Delta |\uu| \geq |\uu|^{p-1} \quad \text{in $\Omega(\uu)$.}\ee
Indeed, given $x_0 \in \Omega(\uu)$, we can assume that $u^1(x_0)= |\uu(x_0)|$. Then,
$$\Delta |\uu|(x_0) \geq \Delta u^1(x_0)= |\uu|^{p-2}u^1(x_0)= |\uu|^{p-1}(x_0).$$
 \end{rem}

Our main theorem for viscosity solutions is as follows. 

\begin{thm}\label{main}
Let $\uu$ be a viscosity solution to \eqref{VOP} in $B_1$. There exists a universal constant $\bar \eps >0$ such that if $\uu$ is $\bar \eps$ flat in $B_1$, 
then $\Gamma(\uu) \in C^{1,\alpha}$ in $B_{1/2}.$
\end{thm}

Viscosity solutions include minimizers of our energy functional. Precisely,

\begin{prop}\label{prop:min-vis} Let $\uu$ be a minimizer to $E(\cdot)$ in $\Omega$. Then $\uu$
 is a viscosity solution to \eqref{VOP}.
\end{prop}



The proof of this Proposition follows as in Remark \ref{notouch}, given that minimizers are $C^{1,\ka-1}$. 
Theorem \ref{main} and Proposition \ref{prop:min-vis} imply the $C^{1,\alpha}$ regularity of minimizing free boundaries and we are left with the question of analiticity to conclude the proof of Theorem \ref{thm:main}.



\subsection{Structure of the paper}
The paper is organized as follows. In Section~\ref{sec:reg}, we obtain the optimal regularity of minimizers. Sections ~\ref{sec:Harnack}-~\ref{sec:flat} are devoted to the proof of our central result, the $C^{1,\alpha}$-regularity of flat free boundaries. Specifically, Section~\ref{sec:Harnack} contains the proof of a Harnack type inequality, Theorem~\ref{holder}. In Section~\ref{sec:flat} we prove the improvement of flatness property for viscosity solutions and establish the regularity of flat free boundary. Finally, in Section~\ref{sec:analy} we extend the $C^{1,\alpha}$-regularity to the analyticity.


\section{Optimal regularity of minimizers}\label{sec:reg}
The objective of this section is to establish the optimal  $C^{1,\ka-1}$ regularity of minimizers. We start by introducing some notation related to energy functional.

For $p\in(0,1)$, $B_r(x_0)\Subset B_1$ and $\uu\in W^{1,2}(B_r(x_0);\R^m)$, we set
\begin{align*}
&E(\uu,x_0,r):=\int_{B_r(x_0)}\left(|\D\uu|^2+\frac{2}{p}|\uu|^p\right),\\
&\tilde E(\uu,x_0,r):=\int_{B_r(x_0)}\left(|\D\uu|^2+\frac{2}{p}|\uu|^2\right).
\end{align*}
When $x_0=0$, we simply write $E(\uu,0,r)=E(\uu,r)$ and $\tilde E(\uu,0,r)=\tilde E(\uu,r)$.
By using 
Hölder's inequality and interpolation inequality, it is easily seen that
\begin{align*}
&E(\uu,1)\le\tilde E(\uu,1)+\tilde E(\uu,1)^{p/2},\\
&\tilde E(\uu,1)\le C(n,p)\left(E(\uu,1)+E(\uu,1)^{2/p}\right).
\end{align*}

\begin{lem}
For $0<p<1$, there exists a constant $C_0=C_0(n,p)>0$ such that \begin{align}
    \label{eq:v-w-p-est}
    \int_{B_r(x_0)}\left(|\vv|^p-|\w|^p\right)\le \frac{p}{4}\int_{B_r(x_0)}|\D(\vv-\w)|^2+C_0r^{n+\frac{2p}{2-p}},
\end{align}
for any ball $B_r(x_0)\subset \R^n$ and for any functions $\vv$, $\w\in W^{1,2}(B_r(x_0);\R^m)$ with $\vv=\w$ on $\p B_r(x_0)$.
\end{lem}

\begin{proof}
For notational simplicity, we may assume $x_0=0$. Since $0<p<1$, we have that for any $a>0$, $b>0$, there holds $(a+b)^p\le a^p+b^p$ (see e.g. Lemma 2.5 in \cite{LeiDeQTei15}), which implies that $|\vv|^p-|\w|^p\le |\vv-\w|^p$. This, combined with Young's inequality and Poincare's inequality, yields that for any $\e\in(0,1)$,
\begin{align*}
    \int_{B_r}\left(|\vv|^p-|\w|^p\right)&\le \int_{B_r}|\vv-\w|^p\\
    &\le\int_{B_r}\left(\e\left(r^{-p}|\vv-\w|^p\right)^{\frac2p}+C(\e)\left(r^p\right)^{\frac2{2-p}}\right)\\
    &=\e r^{-2}\int_{B_r}|\vv-\w|^2+C(n,\e)r^{n+\frac{2p}{2-p}}\\
    &\le\e C_1(n)\int_{B_r}|\D(\vv-\w)|^2+C(n,\e)r^{n+\frac{2p}{2-p}}.
\end{align*}
Now, \eqref{eq:v-w-p-est} follows by taking $\e=\e(n,p)$ small.
\end{proof}

\begin{prop}\label{prop:Mor-est}
Let $\uu$ be a minimizer in $B_1$. There exists a universal constant $C_0>0$, such that \begin{align}\label{eq:Mor-est}
   \tilde E(\uu,x_0,\rho)\le C_0\left[\left(\rho/ r\right)^n\tilde E(\uu,x_0,r)+r^n\right] ,
\end{align}
for any $x_0\in B_{1/2}$ and $0<\rho<r<1/2$.
\end{prop}

\begin{proof}
Without loss of generality, we may assume $x_0=0$. Let $\h\in W^{1,2}(B_r;\R^m)$ be a harmonic function with $\h=\uu$ on $\p B_r$. Then, from the harmonicity of $\h$, $$
\int_{B_r}\mean{\D\h,\D(\h-\uu)}=0.
$$
Combining this with \eqref{eq:v-w-p-est}, we deduce that for $0<r<1/2$ \begin{align*}
    \int_{B_r}|\D(\uu-\h)|^2
    &=\int_{B_r}|\D\uu|^2-|\D\h|^2\le\frac2p\int_{B_r}\left(|\h|^p-|\uu|^p\right)\\
    &\le \frac12\int_{B_r}|\D(\uu-\h)|^2+C(n,p)r^{n+\frac{2p}{2-p}}.
\end{align*}
Applying Poincare's inequality, we further have
\begin{align}
    \label{eq:u-h-est}
    \int_{B_r}\left(|\D(\uu-\h)|^2+|\uu-\h|^2\right)\le Cr^{n+\frac{2p}{2-p}}.
\end{align}
Moreover, we note that $|\h|^2$ and $|\D\h|^2$ are subharmonic, thus satisfy the following sub-mean value property 
\begin{align*}
\int_{B_\rho}\left(|\D\h|^2+\frac2p|\h|^2\right)\le\left(\frac\rho r\right)^n\int_{B_r}\left(|\D\h|^2+\frac2p|\h|^2\right),\quad0<\rho<r.
\end{align*}
We combine this inequality with \eqref{eq:u-h-est} to conclude that \begin{align*}
    &\int_{B_\rho}\left(|\D\uu|^2+\frac2p|\uu|^2\right)\\
    &\qquad\le 2\int_{B_\rho}\left(|\D \h|^2+\frac2p|\h|^2+|\D(\uu-\h)|^2+\frac2p|\uu-\h|^2\right)\\
    &\qquad \le 2(\rho/r)^n\int_{B_r}\left(|\D\h|^2+\frac2p|\h|^2\right)+2\int_{B_\rho}\left(|\D(\uu-\h)|^2+\frac2p|\uu-\h|^2\right)\\
    &\qquad\le 4(\rho/r)^n\int_{B_r}\left(|\D \uu|^2+\frac2p|\uu|^2\right)+6\int_{B_r}\left(|\D(\uu-\h)|^2+\frac2p|\uu-\h|^2\right)\\
    &\qquad\le 4(\rho/r)^n\int_{B_r}\left(|\D \uu|^2+\frac2p|\uu|^2\right)+Cr^{n+\frac{2p}{2-p}}.\qedhere
\end{align*}
\end{proof}

From here, we deduce the almost Lipschitz regularity of $\uu$ with the help of the following lemma, whose proof can be found in \cite{HanLin97}.

\begin{lem}\label{lem:HL}
  Let $r_0>0$ be a positive number and let
  $\vp:(0,r_0)\to (0, \infty)$ be a nondecreasing function. Let $a$,
  $\beta$, and $\gamma$ be such that $a>0$, $\gamma >\beta >0$. There
  exist two positive numbers $\e=\e(a,\gamma,\beta)$,
  $c=c(a,\gamma,\beta)$ such that, if
$$
\vp(\rho)\le
a\Bigl[\Bigl(\frac{\rho}{r}\Bigr)^{\gamma}+\e\Bigr]\vp(r)+b\, r^{\beta}
$$ for all $\rho$, $r$ with $0<\rho\leq r<r_0$, where $b\ge 0$,
then one also has, still for $0<\rho<r<r_0$,
$$
\vp(\rho)\le
c\Bigl[\Bigl(\frac{\rho}{r}\Bigr)^{\beta}\vp(r)+b\rho^{\beta}\Bigr].
$$
\end{lem}

\begin{thm}\label{thm:alm-Lip-reg}
Let $\uu$ be a minimizer in $B_1$. Then $\uu\in C^{0,\si}(B_1)$ for all $0<\si<1$. Moreover, for any $K\Subset B_1$, $$
\|\uu\|_{C^{0,\si}(K)}\le C\left(\tilde E(\uu,1)^{1/2}+1\right)
$$
with $C=C(n,p,\si,K)$.
\end{thm}

\begin{proof}
Without loss of generality, we may assume $K=B_{1/2}$. We fix $0<\sigma<1$. By applying Proposition~\ref{prop:Mor-est} and Lemma~\ref{lem:HL} (with $\e=0$), we have for any $x_0\in B_{1/2}$ and $0<\rho<r<1/2$
$$
\tilde E(\uu,x_0,\rho)\le C(n,p,\sigma)\left[(\rho/r)^{n+2\si-2}\tilde E(\uu,x_0,r)+\rho^{n+2\si-2}\right].
$$
Taking $r\nearrow 1/2$ gives
\begin{align*}
    \int_{B_\rho(x_0)}|\D \uu|^2\le \tilde E(\uu,x_0,\rho)\le C(n,p,\sigma)(\tilde E(\uu,1)+1)\rho^{n+2\si-2}.
\end{align*}
By Morrey space embedding, we conclude $\uu\in C^{0,\si}(B_{1/2})$ with 
\begin{equation*}
\|\uu\|_{C^{0,\si}(B_{1/2})}\le C(n,p,\si)\left(\tilde E(\uu,1)^{1/2}+1\right).\qedhere
\end{equation*}
\end{proof}

We are now ready to prove our first main result, the optimal regularity of the minimizer.

\begin{proof}[Proof of Theorem~\ref{thm:grad-u-holder}] Without loss of generality, we assume $K=B_{1/2}$. For $x_0\in B_{1/2}$ and $0<\rho<r<1/2$, let $\h\in W^{1,2}(B_r(x_0);\R^m)$ be a harmonic function such that $\h=\uu$ on $\p B_r(x_0)$. Note that
$$
\int_{B_\rho(x_0)}|\D\h-\overline{\D\h}_{x_0,\rho}|^2\le \left(\frac\rho r\right)^{n+2}\int_{B_r(x_0)}|\D\h-\overline{\D\h}_{x_0,r}|^2.
$$
Moreover, by Jensen's inequality, \begin{align*}
    &\int_{B_\rho(x_0)}|\D\uu-\overline{\D\uu}_{x_0,\rho}|^2\\
    &\qquad\le 3\int_{B_\rho(x_0)}|\D\h-\overline{\D\h}_{x_0,\rho}|^2+|\D(\uu-\h)|^2+|\overline{\D(\uu-\h)}_{x_0,\rho}|^2\\
    &\qquad\le 3\int_{B_\rho(x_0)}|\D\h-\overline{\D\h}_{x_0,\rho}|^2+6\int_{B_\rho(x_0)}|\D(\uu-\h)|^2,
\end{align*}
and similarly, $$
\int_{B_r(x_0)}|\D\h-\overline{\D\h}_{x_0,r}|^2\le 3\int_{B_r(x_0)}|\D \uu-\overline{\D\uu}_{x_0,r}|^2+6\int_{B_r(x_0)}|\D(\uu-\h)|^2.
$$
Now, we use the above inequalities and \eqref{eq:u-h-est} to obtain 
\begin{align*}
    &\int_{B_\rho(x_0)}|\D\uu-\overline{\D\uu}_{x_0,\rho}|^2\\
    &\qquad\le 3\int_{B_\rho(x_0)}|\D\h-\overline{\D\h}_{x_0,\rho}|^2+6\int_{B_\rho(x_0)}|\D(\uu-\h)|^2\\
    &\qquad\le 3\left(\frac\rho r\right)^{n+2}\int_{B_r(x_0)}|\D\h-\overline{\D\h}_{x_0,r}|^2+6\int_{B_\rho(x_0)}|\D(\uu-\h)|^2\\
    &\qquad\le 9\left(\frac\rho r\right)^{n+2}\int_{B_r(x_0)}|\D \uu-\overline{\D \uu}_{x_0,r}|^2+24\int_{B_r(x_0)}|\D(\uu-\h)|^2\\
    &\qquad\le 9\left(\frac\rho r\right)^{n+2}\int_{B_r(x_0)}|\D\uu-\overline{\D\uu}_{x_0,r}|^2+C(n,p)r^{n+\frac{2p}{2-p}}.
\end{align*}
We apply Lemma~\ref{lem:HL} (with $\e=0$) to get 
\begin{align*}
    \int_{B_\rho(x_0)}|\D\uu-\overline{\D\uu}_{x_0,\rho}|^2&\le C(n,p)\left[\left(\frac\rho r\right)^{n+\frac{2p}{2-p}}\int_{B_r(x_0)}|\D\uu-\overline{\D\uu}_{x_0,r}|^2+\rho^{n+\frac{2p}{2-p}}\right]
\end{align*}
for $0<\rho<r<1/2$. Taking $r\nearrow 1/2$ gives 
\begin{align*}
\int_{B_\rho(x_0)}|\D\uu-\overline{\D\uu}_{x_0,\rho}|^2\le C(n,p)\left(\tilde E(\uu,1)+1\right)\rho^{n+\frac{2p}{2-p}}.
\end{align*}
By Campanato space embedding, we obtain $\D\uu\in C^{0,\frac{p}{2-p}}(B_{1/2})$ with 
\begin{equation*}
\|\D\uu\|_{C^{0,\frac{p}{2-p}}(B_{1/2})}\le C(n,p)(\tilde E(\uu,1)^{1/2}+1).
\end{equation*}
This, along with Theorem~\ref{thm:alm-Lip-reg}, concludes Theorem~\ref{thm:grad-u-holder}.
\end{proof}

\section{A Harnack type inequality}\label{sec:Harnack}

In this Section we prove a Harnack type inequality for viscosity solutions to \eqref{VOP}, as defined in the Subsection \ref{sec:visco}. This is the key tool in the linearization approach utilized to obtain the $C^{1,\alpha}$ regularity of the flat part of the free boundary. As mentioned in the introduction, the viscosity method offers a robust alternative to the development of an epiperimetric inequality for minimizers, which is particularly challenging in this context.

\begin{thm}\label{holder}
There exists a universal constant $\overline{\eps}>0$ such that the following holds: suppose $\uu$ solves \eqref{VOP} in $B_1$, and for some point $x_0 \in B_1(\uu) \cup \Gamma(\uu),$
\be\label{flat_2*} |\uu| \leq u_0(x_n+b_0) \quad \text{in $B_r(x_0) \subset B_1$,} 
\ee 
\be\label{flat_2'} u_0(x_n +a_0)\leq u^1 \quad \text{in $B_r(x_0) \cap \{x_n \geq - a_0\}$}
\ee
and
\be\label{58}|u^i|\leq r^\ka\left( \frac{b_0-a_0}{r}\right)^{5/8} \quad \text{in $B_{r}(x_0)$}, \quad i=2,\ldots, m,\ee with $$b_0 -a_0 \leq  \bar\eps r.$$
Then, for some universal constant $\eta\in(0,1)$, 
\be\label{flat_22} |\uu| \leq u_0(x_n+b_1) \quad \text{ in $B_{\eta r}(x_0)$,}
\ee 
\be\label{flat_22'} u_0(x_n +a_1) \leq u^1 \quad \text{ in $B_{\eta r}(x_0) \cap \{x_n \geq -a_1\}$}
\ee
holds with $$a_0 \leq a_1 \leq b_1 \leq b_0, \quad  b_1 - a_1= (1-\eta)(b_0-a_0).$$
 \end{thm}


To prove this theorem, we need the following technical lemma which will allow us to better estimate the size of $u^1$ in the flatness regime (see \eqref{eq:u^1-lower-bound}).

\begin{lem}\label{lem:subhar}
Let $w\geq0$ be a subharmonic function in $B_1$ such that for some $\eps>0$ small,
$$\{w>0\} \subset \{|x_n| < \eps\}.$$ Then, there is a universal constant $c>0$ such that
$$\|w\|_{L^\infty(B_{1/2})} \leq e^{-\frac{c}{\eps}} \|w\|_{L^\infty(B_1)}.$$
\end{lem}

\begin{proof}
In each ball of radius $2\eps$, the set $\{w=0\}$ covers a percentage of the ball, so by Weak Harnak Inequality, the value of $w$ at the center of the ball is less than $(1-c_0)\|w\|_{L^\infty(B_1)}$ for some $c_0>0$ universal. Thus,
$$\|w\|_{L^\infty(B_{1-2\eps})} \leq (1-c_0)\|w\|_{L^\infty(B_1)},$$ 
and we apply this $\lfloor1/4\e\rfloor$ times, where $\lfloor1/4\e\rfloor$ is the greatest integer less than or equal to $1/4\e$. 
\end{proof}

Now we prove Theorem~\ref{holder}.

\begin{proof}[Proof of Theorem~\ref{holder}] After translation and rescaling of order $\ka$, we can assume that $x_0=0, r=2$, $\uu$ is a solution in a large ball containing $B_2$, and $b_0-a_0=2\eps \leq \bar \eps$. We split the proof into two steps.

\medskip\noindent\emph{Step 1.} In this step, we consider the simple case when $a_0=-\eps$ and $b_0=\eps$. To prove the statement in a ball of radius $\eta$, we distinguish two cases 
\begin{align*}
    &\text{{\it Case A:\qquad }  $|\{|\uu|(x) - u_0(x_n)\le0\}\cap B_{\frac{1}{4}}(\bar x)| \geq \frac 1 2 |B_{\frac{1}{4}}(\bar x)|$, }\\
    &\text{{\it Case B: \qquad }  $|\{|\uu|(x) - u_0(x_n)>0\}\cap B_{\frac{1}{4}}(\bar x)|\ge \frac 1 2 |B_{\frac{1}{4}}(\bar x)|$,}
\end{align*}
where $\bar x= \frac 12 e_n$.

\medskip\noindent 
{\it Case A:} We first consider the case
$$|\{|\uu|(x) - u_0(x_n)\le0\}\cap B_{\frac{1}{4}}(\bar x)| \geq \frac 1 2 |B_{\frac{1}{4}}(\bar x)| .$$
Call $w(x):= u_0(x_n+\eps) -|\uu|(x) \geq 0$. By the above inequality and the Lipschitz regularity of $u_0$, we have for some $c=c(p)>0,$
$$|\{w(x) \geq c \eps \}\cap B_{\frac{1}{4}}(\bar x)| \geq \frac 1 2 |B_{\frac{1}{4}}(\bar x)| .$$
For a small universal constant $\de>0$ with $\e\ll \de$ to be chosen later, by using \eqref{sub2}, \eqref{ODE}, \eqref{flat_2*} and the fact that $p \in (0,1)$, we have that in $B_1 \cap \{x_n>\frac \delta 2\} \subset B_1(\uu)$,
$$
\Delta w(x) = \Delta u_0(x_n+\eps)  -\Delta |\uu|(x) \leq u_0^{p-1}(x_n+\eps) - |\uu|^{p-1}(x) \leq 0.
$$
Then, by Weak Harnack Inequality, 
\be w \geq c_\delta\eps \quad \text{in $B_{\frac{3}{4}}\cap \{x_n \geq \delta\}$},\ee
which is equivalent to
\be\label{top}|\uu|(x)\leq u_0(x_n+\eps) - c_\delta \eps \quad \text{in $B_{\frac 3 4} \cap \{x_n\geq \delta\}$.}\ee

Next, we set
\begin{align}\label{eq:domain}
\begin{split}
    &\cR:=\{|x'| <1/2, \,\,\,-\delta < x_n < \delta\},\qquad \cS:=\{|x'|=\frac 1 2,\,\,\, -\delta< x_n <\delta\},\\
    &\cT:=\{|x'| \leq \frac 1 2,\,\,\, x_n =\delta\}.
\end{split}\end{align}
We wish to construct an explicit nonnegative test function $\psi_0$ in $B_1$ satisfying
\begin{align}
    &\label{super2}\Delta \psi_0 < \psi_0^{p-1} \quad \text{in $\cR \cap \{\psi_0>0\}$},\\
    &\label{B1}\psi_0(x) \geq u_0(x_n+\eps)\quad \text{on $\cS$ with strict inequality on $\cS \cap \{\psi_0>0\}$,}\\
    &\label{B2}\psi_0(x) > u_0(x_n+\eps) - c_\delta \eps\quad\text{on $\cT$, where $c_\de$ is as in } \eqref{top},\\
    &\label{improv}\psi_0(x) \leq u_0(x_n+(1-c)\eps)\quad\text{in $B_\eta$ for some small universal $c, \eta>0$. }
\end{align}
For this purpose, let $\cB$ be the ball of radius $\frac{C_0}{\eps}$ and centered at $-(\frac {C_0} \eps+(1-c_1)\eps)e_n$,  i.e.,
$$
\cB:=B_{\frac {C_0} {\eps}}(-(\frac {C_0} \eps+(1-c_1)\eps)e_n),
$$
and we set
\begin{align}\label{eq:dist-ftn}
d(x): =\begin{cases}\dist (x, \p \cB), & x \not\in\cB,\\ 0,& x\in \cB.\end{cases}
\end{align}
We then define
$$\psi_0(x):= (1-c_0\eps) u_0(d(x)), \quad x \in B_2.$$
Here, $c_0$, $c_1$ and $C_0>0$ are universal constants to be made precise later. Below, we show that such $\psi_0$ satisfies \eqref{super2} - \eqref{improv}.

We first prove the supersolution property \eqref{super2}. It is easily seen that
$$
\Delta d(x)=\frac{n-1}{d(x)+\frac{C_0}\e},\quad x\in \cB^c.
$$
We then compute that in its positivity set (for notational simplicity we omit the dependence on $x$)
\begin{align}\label{comp}
 \Delta \psi_0 & = (1-c_0\eps)(u_0''(d) + \eps \frac{n-1}{\eps d+C_0}u_0'(d))\\
 \nonumber  & =(1-c_0\eps)u_0''(d)( 1 + \eps \frac{n-1}{\eps d+C_0}\frac{u_0'(d)}{u_0''(d)})\\
\nonumber & \leq (1-c_0\eps)u_0''(d)(1 + M \eps \frac{(n-1)}{C_0}),
\end{align}
where we have used that in $B_1 \cap \{d>0\}$
\be\label{u'u''} \frac{u_0'(d)}{u_0''(d)} \leq M 
\ee with $M>0$ universal.
Thus, given that $u_0''(t) = u_0^{p-1}(t)$ and $p \in (0,1),$
$$
\Delta \psi_0 <(1-c_0\eps)^{p-1} u_0(d)^{p-1}=\psi_0 ^{p-1}
$$
as long as 
$$(1-c_0\eps)(1 + M\frac{(n-1) \eps }{C_0}) < 1,$$ 
which is true for \begin{equation*}c_0 C_0 > K:=M (n-1).\end{equation*}
In particular we can pick
\begin{equation}\label{C0}
C_0 = 2 c_0^{-1}K, 
\end{equation} 
and $c_0>0$ will be specified later.

Next, we prove \eqref{B1}. Since $u_0(x_n+\e)=0$ whenever $x_n\le -\e$, we only need to consider points on $\cS \cap \{x_n \geq -\eps\}$. Let
\be\label{c1}c_1:= \frac{1}{32 C_0} = \frac{c_0}{64 K}.\ee
Then a direct computation gives
$$
\dist\left(x,-(\frac{C_0}\e+(1-c_1)\e)e_n\right)>\frac{C_0}\e,\quad x\in \cS\cap\{x_n\ge-\e\},
$$
which implies
$$
(\cS\cap\{x_n\ge -\e\})\cap\cB=\emptyset.
$$
For $x\in \cS\cap\{x_n\ge-\e\}$, we write $d_1: = d(x)$ and $d_2:=\dist(x,\{y:y_n=-\e\})=x_n+\e$. 
Notice that, up to an error of order $\eps^2$, $d_1$ is the vertical distance to the graph of the parabola 
$$
y_n +\eps= c_1\eps - \frac{\eps}{2C_0}|y'|^2,
$$ 
hence
\begin{align}\label{eq:dist-1-2}
d_1=x_n+(1-c_1)\e+\frac{\e}{8C_0}+O(\e^2)\geq d_2 + (\frac{\eps}{16C_0} - c_1\eps).
\end{align}
In view of the definitions of $\psi_0$ and $d_2$, \eqref{B1} follows once we show for $x$ as above
$$(1-c_0\eps) u_0(d_1) > u_0(d_2).$$ 
Due to \eqref{c1} and \eqref{eq:dist-1-2}, it is enough to prove that 
$$(1-c_0\eps) u_0(d_2 + \frac{\eps}{32C_0}) > u_0(d_2).$$
By the strict convexity of $u_0$, it suffices to check that, 
$$(1-c_0\eps) (u_0(d_2) + u_0'(d_2)\frac{\eps}{32C_0}) - u_0(d_2) \geq 0.$$
Due to \eqref{C0}, this holds if we choose $\delta$ small enough such that
$$-u_0(t) + \frac{1}{ 128 K} u_0'(t) \geq 0 \quad \text{on $[0,2\delta].$}$$
Thus, \eqref{B1} is verified for $\delta$ small universal satisfying $2\de\le\frac{\ka}{128K}$.
 
In order to prove \eqref{B2}, it suffices to show that
$$(1-c_0\eps)u_0(\delta+\eps-c_1\eps) - u_0(\delta+\eps)> -c_\delta \eps.$$
Thus, using the Lipschitz continuity of $u_0$ on the interval $[0,2\delta]$, we need to prove that 
$$
(1-c_0\eps)(u_0(\delta+\eps) -c_1 \eps\, c(p,\delta)) - u_0(\delta+\eps)>-c_\delta \eps
$$ 
for a constant $c(p,\delta)>0$. Hence, it suffices to verify that
$$-c_0(u_0(\delta) + O(\eps)) - c_1  c(p,\delta) \geq -c_\delta,
$$
which, from the identity \eqref{c1}, holds true if (for some $c'>0$ universal)
$$-c_0(u_0(\delta) + O(\eps)) - c_0c'(\delta, n, p) \geq -c_\delta,$$
that is, as long as $$c_0 \leq c''(n, p, \delta)$$ for some universal constant $c''>0$.

It remains to prove \eqref{improv}. If $\e>0$ is small, then $\partial\cB$ has a small curvature, thus $d(x)\approx\dist(x_n,\{y_n=-(1-c_1)\e\})$ for $x\in \cB^c$ near the origin. Then we have 
$$d(x) \leq x_n + (1- \frac{c_1}{2})\eps$$
as long as $x \in B_\eta$ with $\eta>0$ universal. Then, in such ball,
$$\psi_0(x) \leq u_0(d(x)) \leq u_0(x_n+(1 -\frac{c_1}{2})\eps),$$
as desired.

\smallskip

We claim that
\begin{align}
    \label{eq:bdry}
    \psi_0\ge|\uu|\text{ on }\partial\cR\quad\text{and}\quad \psi_0>|\uu|\text{ on }\partial\cR\cap\{\psi_0>0\}.
    \end{align}
Indeed, $|\uu|\le u_0(x_n+\e)=0\le \psi_0$ on the bottom of the cylinder $\{|x'|\le1/2,\,\,\, x_n=-\de\}$. This, together with \eqref{top}, \eqref{B1} and \eqref{B2}, implies \eqref{eq:bdry}.

Now we show via a  sliding method that $$\psi_0 \geq |\uu| \quad \text{in $\cR$}.$$
To this end, we consider a family of translates - monotone in $t$ - given by 
$$
\psi_t(x): = \psi_0 (x + t e_n),\quad t\in(-3/4,1),\quad x\in\cR.
$$
We have in $\overline{\cR}$ 
\begin{equation}\label{T}
\psi_1(x)=(1-c_0\e)u_0(d(x+e_n))>(1-c_0\e)u_0(1-\de)>u_0(2\de)\ge u_0(x_n+\e).
\end{equation}
Moreover, if $x\in\overline{\cR}$, then $x-\frac12e_n\in B_1\cap\cB$, thus $d(x-\frac12 e_n)=0$, which implies that
$$
\psi_{-1/2}=0\quad\text{in }\overline\cR.
$$
This, along with \eqref{T} and the monotonicity of the family $\psi_t$, gives 
$$
\bar t:= \inf\{t \in (-3/4,1) : \psi_t \geq |\uu|\quad \text{in $\bar \cR \cap \overline{\{|\uu|>0\}}$}\, \}\in (-1/2,1).
$$
We take a point $\bar x \in \bar \cR \cap \overline{\{|\uu|>0\}}$ satisfying 
$$
\psi_{\bar t}(\bar x) = |\uu|(\bar x).
$$ 
We claim that
\begin{align}
    \label{eq:claim}
    \text{if $\bar t>0$, then $\bar x\in \Gamma(\uu)$.}
\end{align}
To prove it, we assume to the contrary $\bar x\in \{|\uu|>0\}$ (and $\psi_{\bar t}>|\uu|$ on $\Gamma(\uu)$). Note that from \eqref{eq:bdry} and $\psi_{\bar t}\ge\psi_0$, we see that $\bar x\not\in\partial\cR$, thus $\bar x$ is contained in the open set $\cR\cap\{|\uu|>0\}$. Then we have by \eqref{sub2} and \eqref{super2} that
$$
\Delta(\psi_{\bar t}-|\uu|)<\psi_{\bar t}^{p-1}-|\uu|^{p-1}\le 0\quad\text{near }\bar x.
$$
However, $\psi_{\bar t}-|\uu|$ has a local minimum $0$ at $\bar x$, which is a contradiction by the strong maximum principle.

We consider a family of positive functions $\psi_0^\la:B_2\to(0,\infty)$, where $0<\la<\la_0$ for some $\la_0>0$, satisfying
\begin{align}
    \label{inf}
    \begin{cases}
    \Delta \psi_0^\la<(\psi_0^\la)^{p-1}\quad\text{in }\cR,\\
    \text{$\psi_0^\la$ is nondecreasing in $x_n$-variable},\\
    \|\psi_0^\la-\psi_0\|_{L^\infty(\cR)}\to0\quad\text{as }\la\to0,
\end{cases}\end{align}
whose existence will be proved later. Since the inequalities in \eqref{B1}, \eqref{B2} and \eqref{T} are strict, these estimates hold also for $\psi_0^\la$ if $\la>0$ is small enough. For such $\la$, we apply the sliding method above to \begin{equation*}\psi_t^\la(x): = \psi_0^\la (x + t e_n) >0\end{equation*} 
to deduce that no touching can occur on the free boundary $F(\uu)$, which combined with \eqref{eq:claim} gives 
\begin{equation}\psi_0^\la \geq |\uu| \quad \text{in $\cR$}.\end{equation}
Passing to the limit for $\la \to 0$, we obtain $\psi_0\ge |\uu|$ in $\cR$. By combining this with \eqref{improv} and making $\eta>0$ smaller if necessary so that $B_\eta\subset\cR$, we conclude \eqref{flat_22}.

It remains to prove that the family $\psi_0^\la$ exists. To this point, we call $u_\la$ the solution to the ODE:
$$u_\la '' = u_\la^{p-1}\chi_{\{u_\la>0\}}, \quad t \in [0,1],$$
with initial condition
$$u_\la(0) = \la>0, \quad u_\la'(0)=0,$$
and we reflect it evenly across $t=0.$ It is easy to see that $u_\la>0$ and $u''_\la > 0$, hence we have in the interval $[0,1]$ that $u_\la$ and $u'_\la$ are increasing while $u''_\la$ is decreasing. This, along with \eqref{u'u''} and the uniform convergence $u_\la\to u_0$, gives that for $\la>0$ small  
$$u_\la' \leq  N u_\la '' \quad \text{in $[-1,1],$}$$ for some universal constant $N>0$. Thus, if we define 
$$\psi_0^\la(x):= (1-c_0\eps) u_\la(d(x)) >0, \quad x \in B_2,$$
where $d(x)$ is as in \eqref{eq:dist-ftn}, then we can repeat the calculations in \eqref{comp} and conclude that $\psi_0^\la(x)$ is a family of strictly positive supersolutions in $B_1$. Clearly, $\psi_0^\la$ satisfies \eqref{inf} and we are done.

\

\noindent
{\it Case B:} Suppose
$$|\{|\uu|(x) - u_0(x_n)>0\}\cap B_{\frac{1}{4}}(\bar x)| \geq \frac 1 2 |B_{\frac{1}{4}}(\bar x)| .$$
First, observe that by our assumptions, $u^1$ is bounded above and below by universal positive constants in the ball $B_{1/4}(\bar x).$ Thus, using assumption \eqref{58}, we have
\begin{align}\label{eq:u^1}
|\uu| = u^1\sqrt{1+\sum_{i=2}^m(u^i/u^1)^2} = u^1 + O(\eps^{5/4}) \quad \text{in $B_{1/4}(\bar x)$}.
\end{align}
Then, for $\eps>0$ small, the function
$$
w(x):= u^1(x)-u_0(x_n-\eps)
$$ 
satisfies 
\be\label{wh}
|\{w \geq c\eps\}\cap B_{\frac{1}{4}}(\bar x)| \geq \frac 1 2 |B_{\frac{1}{4}}(\bar x)|
\ee 
for some universal $c>0$. Moreover, for a universal constant $\de>0$ to be chosen later, we have by \eqref{super}, \eqref{ODE} and $p\in(0,1)$ that in $B_1\cap \{x_n \geq \delta/2\}$,
$$
\Delta w(x)= \Delta u^1(x) - \Delta u_0(x_n-\eps) \leq (u^1)^{p-1}(x) - u_0^{p-1}(x_n-\eps) \leq 0.
$$ 
From this and \eqref{wh}, we have by applying Weak Harnack inequality,
\be\label{wwbarx}w \geq c_\delta\eps \quad \text{in $B_{\frac{1}{2}}\cap \{x_n \geq \delta\}$},\ee
from which we infer that 
\be\label{wbarx}u^1(x) \geq u_0(x_n-\eps) + c_\delta\eps \quad \text{in $B_{\frac{1}{2}}\cap \{x_n \geq \delta\}$}.\ee

We let $\cR$, $\cS$ and $\cT$ be as in \eqref{eq:domain}. We now construct an explicit test function $\psi_0 \in C^{1,\ka-1}(B_2)$ such that 
\begin{align}
    &\label{test}\Delta \psi_0 > \psi_0^{p-1}\chi_{\{\psi_0>0\}}\,\,\text{ and }\,\,\psi_0 \in C^2 \quad\text{in $B_{3/4} \cap \{\psi_0 \neq 0\}$},\\
    &\label{eq:deriv-positive} \p_n \psi_0 >0 \quad \text{in $B_{3/4}$},\\
    &\label{psiu0}\psi_0(x) < u_0(x_n-\eps)  \quad \text{on $\cS \cap \{x_n \geq \eps\},$}\\
    &\label{eps3} \psi_0(x) \leq -\eps^3  \quad \text{on $\cS \cap \{x_n < \eps\},$}\\
    &\label{psiC0} \psi_0(x) < u_0(x_n-\eps) + c_\delta \eps  \quad \text{on $\cT$,}\\
    &\label{improv2}\psi_0(x) \ge u_0(x_n-(1-c)\eps)\quad\text{in $B_\eta\cap\{x_n\ge (1-c)\e\}$,}
\end{align}
where $c_\de$ is as in \eqref{wbarx} and $c,\eta$ are small universal constants to be chosen later.

To this end, we set
$$
\cB:=B_{\frac {C_0} \eps}((\frac {C_0} \eps+(1-c_1)\eps)e_n)
$$
for some universal constants $C_0, c_1>0$ to be made precise later. Moreover, we let $d(x)$ 
be the signed distance from $x$ to $\p \cB$ which is positive inside of $\cB$, i.e.,
\begin{align*}
    d(x)=\begin{cases}
        \dist(x,\partial\cB)&\text{when }x\in\cB,\\
        -\dist(x,\partial\cB)&\text{when }x\not\in\cB.
    \end{cases}
\end{align*}
We then define
$$
\Psi(t):= u_0(t) + \eps \frac{K}{C_0}(t+ \frac 1 2t^2)
$$
for some universal constant $K>0$ to be determined later and 
$$
\psi_0(x):=\Psi(d(x)).
$$
From $u_0\in C^{1,\ka-1}$ and $d\in C^{1,1}$, we see that $\psi_0\in C^{1,\ka-1}(B_2)$.

We now verify that $\psi_0$ satisfies \eqref{test}-\eqref{improv2}.

We first prove \eqref{test}. In $B_{3/4}$, $\psi_0\neq 0$ if and only if $d\neq0$, and $u_0\in C^2(\R\setminus\{0\})$. Thus $\psi_0\in C^2(B_{3/4}\cap\{\psi_0\neq0\})$. For the subsolution property of $\psi_0$, we observe that $\psi_0$ and $d$ have the same sign in $B_{3/4}$. Moreover,
$$
\Delta d(x)=-\frac{n-1}{\frac{C_0}\e-d(x)}\quad\text{in }\{d>0\}.
$$
Thus, for $K$ large universal and $\e$ small, we have in $\{d>0\}$,
\begin{align*}
    \Delta \psi_0(x)&= \Psi''(d) - \eps \frac{n-1}{C_0 - \eps d} \Psi'(d)\geq u_0^{p-1}(d)+ \frac{K}{C_0}\eps - \eps(n-1) \frac{u_0'(d)}{C_0} +O(\eps^2)\\
    &> u_0^{p-1}(d) \geq \psi_0^{p-1}(x).
\end{align*}
Similarly, in $\{d\le0\},$
$$\Delta \psi_0(x)= \Psi''(d) - \eps \frac{n-1}{C_0 - \eps d}\Psi'(d) \geq  \frac{K}{C_0}\eps +O(\eps^2)> 0.
$$

\eqref{eq:deriv-positive} simply follows from the fact that $\Psi$ is increasing in $(-1,\infty)$ and $d$ is increasing in $x_n$-variable in $B_{3/4}$.

In order to prove \eqref{psiu0}, notice that, as in Case A, \be\label{d}d(x)= x_n -(1-c_1)\eps - \frac{\eps}{2C_0}|x'|^2 + O(\eps^2),\ee hence for $c_1= (32 C_0)^{-1}$,
\begin{align}\label{eq:d}
d(x) \leq x_n - \eps -\frac{\eps}{16C_0} \quad \text{on $\cS$}.
\end{align}
Due to the monotonicity of $\Psi$ in $(-1,\infty)$, it is enough to show that
$$\Psi(x_n-\eps  -\frac{\eps}{16C_0}) < u_0(x_n-\eps) \quad \text{on $\{\eps\le x_n\le\de\}$},$$ 
or alternatively
$$\Psi(t)< u_0(t+\frac{\eps}{16C_0}), \quad -\frac{\eps}{16C_0}\le t\le \de-\e-\frac{\e}{16C_0}.$$
If $t \in [-\frac{\eps}{16C_0}, 0)$, then the inequality is obvious as the left hand side is negative. When $t=0$, we have $\Psi(0)=0<u_0(\frac{\e}{16C_0})$. So it is enough to show the inequality above in the interval  $(0, \delta].$
Using the convexity of $u_0$ and $\ka\in(1,2)$,
\begin{align*}
    &u_0(t+\frac{\e}{16C_0})-\Psi(t)\\
    &=u_0(t+\frac{\e}{16C_0})-u_0(t)-\e\frac{K}{C_0}(t+\frac{t^2}2)\ge\frac{\e}{16C_0}u_0'(t)-\e\frac{K}{C_0}(t+\frac{t^2}2)\\
    &=\frac{\e}{C_0}(\frac{c_p\ka}{16}t^{\ka-1}-K(t+\frac{t^2}2))>0\quad\text{in }(0,\de]
\end{align*}
if $0<\de\ll 1/K$ is sufficiently small universal.

We now check that \eqref{eps3} holds as well. Arguing as above, it suffices to show that 
$$\Psi(t) \leq -\eps^3, \quad -2\delta \leq t \leq -\frac{\eps}{16 C_0}.$$ This holds for small $\e>0$ as $\Psi$ is increasing and $\Psi(-\frac{\eps}{16 C_0}) = -\e^2\frac{K}{16C_0^2}(1-\frac{\e}{32C_0}).$

We now turn to \eqref{psiC0}. From $\cT\subset\cB$, we see that
$$
d(x) \leq x_n -(1-c_1)\eps\quad\text{on }\cT.
$$ 
Thus it is sufficient to show that
$$\Psi(x_n -(1-c_1)\eps) < u_0(x_n-\eps)+c_\delta \eps\quad\text{on }\cT.$$ 
This follows by using the formula for $\Psi$, $c_1=(32C_0)^{-1}$ and the convexity of $u_0$: if $C_0$ is sufficiently large so that $C_0\gg1/\de\gg1/K$, then we have on $\cT$
\begin{align*}
    &\Psi(x_n-(1-c_1)\e)-u_0(x_n-\e)=\Psi(\de-(1-c_1)\e)-u_0(\de-\e)\\
    &\le u_0(\de-\e+c_1\e)-u_0(\de-\e)+\frac{2\de K}{C_0}\e\le c_1\e u_0'(\de-\e+c_1\e)+\frac{2\de K}{C_0}\e\\
    &=\left(\frac{u_0'(\de-\e+c_1\e)}{32}+2\de K\right)\frac{\e}{C_0}< c_\de\e.
\end{align*}

Finally, for \eqref{improv2}, we observe that if $\e>0$ is small so that $\partial\cB$ has a small curvature, then 
$$
d(x)\ge x_n-(1-\frac{c_1}2)\e\ge0\quad\text{in }B_\eta\cap\{x_n\ge(1-\frac{c_1}2)\e\}
$$
with $\eta>0$ small universal. Then,
$$\Psi(d) \geq u_0(d) \geq u_0(x_n - (1- \frac{c_1}{2})\eps)\quad\text{in }B_\eta\cap\{x_n\ge(1-\frac{c_1}2)\e\},$$
as desired.

Next, we claim that
\begin{align}
    \label{eq:bdry-2}
    \psi_0< u^1\quad\text{on }\partial\cR.
\end{align}
Indeed, from \eqref{int1}, we see that $(u^1)^-:=\max\{-u^1,0\}$ is a nonnegative function in $B_1$ and is subharmonic in its positivity set. Thus it is subharmonic in $B_1$. Thanks to \eqref{flat_2*} and \eqref{flat_2'}, we can apply Lemma~\ref{lem:subhar} to obtain that for $\e>0$ small
\begin{align*}
    \|(u^1)^-\|_{L^\infty(B_{3/4})}&\le e^{-c/\e}\|(u^1)^-\|_{L^\infty(B_1)}\le e^{-c/\e}\|u_0(x_n+\e)\|_{L^\infty(B_1)}\\
    &\le Ce^{-c/\e}<\e^{3},
\end{align*}
thus 
\begin{align}
    \label{eq:u^1-lower-bound}
    u^1>-\e^3\quad\text{in }B_{3/4}.
\end{align}
Moreover, from \eqref{flat_2*}, we deduce that in the bottom of $\cR$, $\{|x'|<1/2,\,\,\, x_n=-\de\}$, $|\uu|\le u_0(-\de+\e)=0$, thus $u^1=0>\psi_0$ by the definition of $\psi_0$. This, along with \eqref{flat_2'}, \eqref{wbarx}, \eqref{psiu0}, \eqref{eps3}, \eqref{psiC0} and \eqref{eq:u^1-lower-bound}, implies \eqref{eq:bdry-2}.

Now we use a sliding argument to obtain
\begin{align}\label{eq:psi_u^1}
\psi_0\le u^1\quad\text{in }\cR.
\end{align}
For this purpose, we set
$$
\psi_s(x):=\psi_0(x-se_n),\quad s\in (-3/4,1),\,\,\, x\in\cR.
$$
We observe that by \eqref{eq:u^1-lower-bound}
$$
\psi_{2\de}\le \Psi(-\de)<-\e^3<u^1\quad\text{in }\cR,
$$
and that by \eqref{flat_2*}
$$
\psi_{-1}\ge u_0(x_n+1)>u_0(x_n+\e)\ge|\uu|\ge u^1\quad\text{in }\cR.
$$
Then we have
\begin{align*}
    \bar s:=\inf\{s\in(-1,2\de)\,:\, \psi_s\le u^1\,\,\,\text{in }\bar\cR\cap\overline{\{|\uu|>0\}}\}\in(-1,2\de).
\end{align*}
We claim that $\bar s\le0$, which readily implies \eqref{eq:psi_u^1} as $\psi_0\le \psi_{\bar s}$. To prove the claim, towards a contradiction we suppose $\bar s>0$. We take $\bar x\in\bar\cR\cap\overline{\{|\uu|>0\}}$ such that $\psi_{\bar s}(\bar x)=u^1(\bar x)$. Since $\psi_{\bar s}\le\psi_0$ in $B_{3/4}$, we have by \eqref{eq:bdry-2} that $\bar x\not\in\partial\cR$. Then we have that $\psi_{\bar s}$ touches $u^1$ from below at $\bar x\in\cR\cap\overline{\{|\uu|>0\}}$. This is a contradiction by the definition of viscosity solution \eqref{VOP} and Remark~\ref{notouch}.

Finally, \eqref{flat_22'} follows from \eqref{improv2} and \eqref{eq:psi_u^1}.

\medskip\noindent\emph{Step 2}. In this step, we consider the general case without the assumption $a_0=-\e$ and $b_0=\e$. Three instances can occur. 

If $|a_0| \leq \eta/4$ and $\bar \eps$ is small, then in view of Step 1, we obtain the desired conclusion in $B_\eta(z)$, where $z$ is the middle point between $-b_0$ and $-a_0$. As $B_{\eta/2}\subset B_\eta(z)$, this in turn implies that the statement holds in $B_{\eta/2}$. 

If instead $a_0 < -  \eta/4$, then $|\uu|=0$ in $B_{\eta/8}$ by \eqref{flat_2*} and the conclusion is trivial. 

Finally, if $a_0> \eta/4$, then $B_{\eta/8} \subset \{u^1>0\} \cap \{|\uu|>0\}$ by \eqref{flat_2'}, thus we can use the standard weak Harnarck inequality as in the beginning of Cases A and B in Step 1. 
\end{proof}

By restricting the domain of $u_0$ from $\R$ to $[0,\infty)$, we can consider its inverse function $u_0^{-1}:[0,\infty)\to[0,\infty)$. Then a repeated application of Theorem \ref{holder} leads to the following corollary. 

\begin{cor}\label{AA} Let $\uu$ be a solution to  \eqref{VOP} in $B_1$ such that for $\eps>0$
\be\label{flat_2} |\uu| \leq u_0(x_n+\eps) \quad \text{in $B_1$,} 
\ee 
\be\label{flat_2''} u_0(x_n -\eps)\leq u^1 \quad \text{in $B_1 \cap \{x_n \geq \eps\}$}
\ee
and
$$|u^i|\leq \eps^{3/4} \quad \text{in $B_1$}, \quad i=2,\ldots, m.$$Denote
\be\label{tilde}
\widetilde u^1= \frac{u_0^{-1}((u^1)^+) - x_n}{\eps}, \quad \widetilde{|\uu|} = \frac{u_0^{-1}(|\uu |)-x_n}{\eps},
\ee defined in $\overline{\{|\uu|>0\}}.$

 There exists a small universal constant $\bar \eps>0$ such that if $\eps \leq \bar \eps$, then $\widetilde u^1$ and $\widetilde{|\uu|}$ have a universal H\"older modulus of continuity at $x_0 \in B_{1/2} \cap \{|\uu|>0\}$ outside a ball of radius $r_\eps,$ with $r_\eps \to 0$ as $\eps \to 0$, i.e., for any $x_0\in B_{1/2}\cap\{|\uu|>0\}$,
 \begin{align*}
     &|\widetilde u^1(x)-\widetilde u^1(x_0)|\le C|x-x_0|^\beta,\quad x\in\{|\uu|>0\}\setminus B_{r_\e}(x_0),\\
     &|\widetilde{|\uu|}(x)-\widetilde{|\uu|}(x_0)|\le C|x-x_0|^\beta,\quad x\in \{|\uu|>0\}\setminus B_{r_\e}(x_0),
 \end{align*}
 where $\beta\in(0,1)$ and $C>0$ are universal constants.
\end{cor}

\begin{proof}
We apply Theorem \ref{holder} in balls $B_{r_k}(x_0)$ with $r_k= \eta^k$, $k=1,2,\ldots$ and obtain
\be\label{coro}u_0(x_n + a_k) \leq (u^1)^+ \leq |\uu| \leq u_0(x_n+b_k) \quad \text{in $B_{r_k}(x_0)$}\ee
with \be\label{osci}b_k-a_k = 2\eps (1-\eta)^{k-1} = c\eps r_k^\beta ,\quad |a_k|, |b_k| \leq \eps,\ee
where $\beta=\log_\eta(1-\eta)\in(0,1)$, as long as $r_k$ satisfies the inequality
$$
|u^i| \leq \eps^{3/4} \leq r_k^\ka \left(\frac{b_k-a_k}{r_k}\right)^{5/8}.
$$ 
This inequality holds for $k\leq C |\log \eps|,$ which tends to $\infty$ as $\eps \to 0.$ It is easy to check that inequality \eqref{coro} implies that $\widetilde u^1, \widetilde{|\uu|}$ are trapped between $a_k/\eps$ and $b_k/\eps$ in $B_{r_k}(x_0) \cap \overline{\{|\uu|>0\}}$, and in view of \eqref{osci} their oscillation is controlled by $cr_k^\beta.$
\end{proof}


\section{Improvement of Flatness \\ (Proof of Theorem \ref{main})}\label{sec:flat}

In this section, we will prove our main Improvement of Flatness lemma. This is the key tool from which the desired regularity Theorem \ref{main} is derived from a standard iterative argument (see for example \cite{DeSTor20}).

\begin{lem}[Improvement of Flatness]\label{IMPF} Let $\uu$ be a viscosity solution to \eqref{VOP} in $B_1$ satisfying the following version of $\eps$-flatness assumption in $B_1$
\be\label{flat1} |\uu| \leq u_0(x_n+\eps) \quad \text{in $B_1$},
\ee
\be\label{flat1'} u^1 \geq u_0(x_n-\eps) \quad \text{in $B_1 \cap \{x_n \geq \eps\}$},
\ee
and
\be\label{34} |\uu - f^1u^1| \leq \eps^{3/4}  \quad \text{in $B_1$},
\ee
with $0 \in \Gamma(\uu).$ If $0<r \leq r_0$ for a universal constant $r_0>0$, and $0<\eps \leq \eps_0$ for some $\eps_0$ depending on $r$, then
\be\label{flat_imp}
|\uu| \leq u_0(\langle x, \nu\rangle + \eps \frac r  2) \quad \text{in $B_r$},
\ee
\be\label{non_dimpr} \langle \uu, \bar f\rangle \geq u_0(\langle x, \nu\rangle - \eps \frac r  2) \quad \text{in $B_r \cap \left\{\langle x , \nu \rangle \geq  \eps \frac{r}{2}\right\}$}, \ee
and
\be\label{new} |\uu - (\langle \uu, \bar f \rangle) \bar f| \leq (\frac{\eps}{2})^{3/4}r^\ka \quad \text{in $B_r$,} 
\ee
with $|\nu -e_n|\leq C\eps, |\bar f-f^1| \leq C \eps$ for a universal constant $C>0.$
\end{lem}

The following lemma is straightforward and it guarantees that we can replace the flatness assumptions \eqref{flat}-\eqref{nond} in our main Theorem \ref{main}, with \eqref{flat1}-\eqref{flat1'} and \eqref{34}. Thus, the conclusion of Theorem \ref{main} follows by standard arguments from Lemma \ref{IMPF} (see for example \cite{DeSTor20}).


\begin{lem}\label{change}Let $\uu$ be a viscosity solution to  \eqref{VOP} in $B_1$ such that for $\eps>0$ small universal,
\be\label{flat_harnack1} |\uu- f^1 u_0(x_n)| \leq \eps \quad \text{in $B_1$,}
\ee
and
\be\label{non_d11} |\uu| \equiv 0 \quad \text{in $B_1 \cap \{x_n < - \eps\}$}. \ee Then
 \be  |\uu| \leq u_0(x_n+\eps^{\frac{1}{2\ka}}) \quad \text{in $B_1$}\ee
 and
\be u_0(x_n-\eps^{\frac{1}{2\ka}})\leq u^1  \quad \text{in $B_1 \cap \{x_n \geq \eps^{\frac{1}{2\ka}}\}$}.\ee
\end{lem}

This next lemma will be used in the proof of Lemma \ref{IMPF}.
\begin{lem}\label{54}Let $\uu$ be a viscosity solution to  \eqref{VOP} in $B_1$ such that for $\eps>0$,
$$|u^i| \leq \eps^{3/4} \quad \text{in $B_1$}, \quad i=2,\ldots,m,$$ 
and 
\be\label{no} |\uu| \equiv 0 \quad \text{in $B_{1} \cap \{x_n < - \eps\}$}. \ee
Then, for $C>0$ universal, 
\be |u^i| \leq C\eps^{3/4}(x_n+\eps)^+ \quad \text{in $B_{1/2}$}, \quad i=2,\ldots,m.  \ee
\end{lem}
\begin{proof} 
Notice that for $i=2,\ldots,m$, $(u^i)^\pm$ are both subharmonic in $B_1 \cap \{x_n > -\eps\}$ by \eqref{int1} and vanish on $B_1\cap\{x_n=-\eps\}$. Let $w:B_1\cap\{x_n>-\e\}\to\R$ be the solution of
\begin{align*}
    \begin{cases}
        \Delta w=0&\text{in }B_1\cap\{x_n>-\e\},\\
        w=0&\text{on }B_1\cap\{x_n=-\e\},\\
        w=1&\text{on }\partial B_1\cap\{x_n\ge-\e\}.
    \end{cases}
\end{align*}
Note that by the boundary Harnack principle, $w(x)\le C(x_n+\e)$ in $B_{1/2}\cap\{x_n>-\e\}$ for some universal constant $C>0$. Moreover, $(u^i)^\pm\le \e^{3/4}w$ in $B_1\cap \{x_n>-\e\}$ by the comparison principle. These two inequalities, along with \eqref{no}, conclude the lemma.
\end{proof}

To prove Lemma~\ref{IMPF}, we also need to consider the linearized problem
\begin{equation}\label{LiE}
\begin{cases}
\Delta \varphi + s\dfrac{ \varphi_n}{x_n}=0& \text{in $B_{1}^+,$}\\
\frac{\partial }{\partial x_n^{1-s}} \, \varphi =0& \text{on $B_1'$},
\end{cases}
\end{equation}
with $s>-1.$

We say that $\varphi: \overline {B}_1^+ \to \mathbb R$ is a viscosity solution of the equation above if it satisfies the equation in $B_1^+$ in the viscosity sense and  

a) If $s \geq 1$, $\varphi$ is uniformly bounded in $B_1^+$;

b) If $s<1$, $\varphi$ cannot be touched by below (above) at a point on $B_1'$ by a test function
$$q(x):= a |x'- y'|^2 + b + \gamma x_n^{1-s}, \quad \quad a,b \in \mathbb R, y' \in \mathbb {R}^{n-1},$$
with $\gamma>0$ ($\gamma<0$). 

\

In \cite{DeSSav21}, the following $C^{1,\sigma}$ estimate was proved for solutions of \eqref{LiE}.

\begin{thm}\label{T7.2}
Assume that $\varphi$ is a solution of \eqref{LiE}, and $s > -1$. Then
$$|\varphi(x)- \varphi(0) - a' \cdot x' | \le C \|\varphi\|_{L^\infty} \,  |x|^{1+\sigma}\quad\text{in }B_{1/2}^+,$$
with $C$ large and $\sigma >0$ small, depending only on $n$ and $s$.
\end{thm}

We now provide the proof of the improvement of flatness lemma by following the strategy of \cite{DeS11}.

\begin{proof}[Proof of Lemma \ref{IMPF}]
We divide the proof into four different steps.

\medskip\noindent\emph{Step 1 (Compactness):} Fixing $r \leq r_0$ with $r_0$ universal (the value of $r_0$ will be given in Step 4), suppose by contradiction that there exists $\eps_k \to 0$ and a sequence of solutions $\{\uu_k\}_k$ of \eqref{VOP} such that $0 \in \Gamma(\uu_k)$ and \eqref{flat1}-\eqref{flat1'} and \eqref{34} are satisfied for every $k$,
but  the conclusions \eqref{flat_imp}-\eqref{non_dimpr} and \eqref{new} of the lemma do not hold.\\
Let us set in $B_1(\uu_k) \cup \Gamma(\uu_k) \subset \{x_n  \geq -\eps_k\}$
\be\label{tildek}
\widetilde u^1_k= \frac{u_0^{-1}((u^1_k)^+) - x_n}{\eps_k}, \quad \widetilde{|\uu|}_k = \frac{u_0^{-1}(|\uu_k |)-x_n}{\eps_k }.
\ee
By the flatness assumptions \eqref{flat1}-\eqref{flat1'}, $\{\widetilde u^1_k\}_k$ and $\{\widetilde{|\uu|}_k\}_k$ are uniformly bounded in $B_1$. Moreover, $\Gamma(\uu_k)$ converges to $B_1'$ in the Hausdorff distance. Now, by Corollary \ref{AA} and Ascoli-Arzel\`{a}, it follows that, up to a subsequence, the graphs of $\widetilde{u}^1_k$ and of $\widetilde{|\uu|}_k$ over $B_{1/2} \cap (B_1(\uu_k) \cup \Gamma(\uu_k))$ converge in the Hausdorff distance to the graphs of Hölder continuous functions $\widetilde u^1_\infty, \widetilde{|\uu|}_\infty$ in $B_{1/2}^+\cup B_{1/2}'$. Moreover, for any $0<\de<1/2$, by using the flatness assumptions and arguing as in \eqref{eq:u^1}, we have for small $\e_k$,
$$||\uu_k| - u^1_k| \leq C(n,\delta) \eps_k^{3/2} \quad \text{in $B_{1} \cap \{x_n > \delta\}$},$$
with $C> |\uu|\geq u^1_k>c \delta^\ka$ for some universal constants $C,c>0$. Thus, given that $u_0^{-1}$ is Lipschitz on compact intervals in $(0,\infty)$, we conclude that  
$$|\widetilde{|\uu|}_k - \widetilde{u}_k^1| \leq C'(n,\delta) \eps_k^{1/2} \quad \text{in $B_{1} \cap \{x_n > \delta\}$},$$ hence in the limit
\be\label{equal} \widetilde{|\uu|}_\infty \equiv \widetilde{u}_\infty^1 \quad \text{in $B_{1/2} \cap \{x_n \geq 0\}.$}\ee

\medskip\noindent\emph{Step 2 (Linearized problem):} In this step, we show that $v:=\widetilde {|\uu|}_\infty$ satisfies the linearized problem \eqref{LiE} in $B^+_{1/2}$ in the viscosity sense for $s:=2(\ka-1)\in(0,2)$.

\medskip\noindent\emph{Step 2-1.} We start with the interior equation. We show that the equation is satisfied in the viscosity sense. Assume that $\varphi \in C^2$ touches $v$ strictly by above at $x_0 \in B^+_{1/2}.$ Then, up to adding a small constant to $\varphi$, $u_0(x_n + \eps_k \varphi)$ touches $|\uu_k|$ by above 
at a point $x_k \to x_0.$ Note that $x_k\in B_{1/2}\cap\{x_n>\e_k\}\subset\{|\uu_k|>0\}$ for large $k$. For such $k$, since in a neighborhood of $x_k$,
$$\Delta |\uu_k| \geq |\uu_k|^{p-1},$$ 
then
$$\Delta (u_0(x_n+\eps_k \varphi)) \geq u_0^{p-1}(x_n+\eps_k \varphi).$$ 
Thus, we have for $g(t)=t^{p-1}$
\begin{align*}
    \Delta(u_0(x_n+\e_k\vp))-\Delta u_0(x_n)&\ge g(u_0(x_n+\e_k\vp))-g(u_0(x_n))\\
    &=\e_kg'(u_0(x_n))u_0'(x_n)\vp+O(\e_k^2).
\end{align*}
By dividing by $\e_k$, we get that in the limit
$$\Delta (u_0'(x_n)\varphi) \geq g'(u_0(x_n))u_0'(x_n)\varphi \quad \text{at $x_0$.}$$ 
Combining this with 
$$\Delta (u_0'(x_n)) = g'(u_0(x_n))u_0'(x_n),$$ 
which follows by differentiating the equation satisfied by $u_0$ in the $x_n$-direction, we derive that 
$$\Delta \varphi + 2 \frac{\nabla u_0'(x_n)}{u_0'} \cdot \nabla \varphi \geq 0 \quad \text{at $x_0$,}$$
hence by homogeneity
$$\Delta \varphi + 2(\ka-1)\frac{e_n}{x_n} \cdot \nabla \varphi \geq 0 \quad \text{at $x_0$.}$$
If $\varphi$ touches $v$ strictly by below, we repeat the above process with $u_k^1$ rather than $|\uu_k|,$ since $$\Delta u^1_k \leq (u^1_k)^{p-1}\quad\text{in }B_{1/2}\cap\{x_n>\e_k\}.$$

\medskip\noindent\emph{Step 2-2.} Since $v$ is bounded in $B^+_{1/2}$, it is sufficient to check the boundary condition on $B_{1/2}'$ for $s\in(0,1)$.

\medskip\noindent\emph{Step 2-2-1.}
Assume by contradiction that there exists a function 
$$q(x):=  \frac 1 2 a |x'- y'|^2 + b + \gamma x_n^{1-s}, \quad \quad a,b \in \mathbb R, y' \in \mathbb {R}^{n-1},$$
with $\gamma<0$, which touches $v$ by above at $x_0 \in B_{1/2}'.$ Without loss of generality we can assume $a \geq 1$, and since $s\in (0,1)$, we can replace $q$ with the function
$$\frac 1 2 a |x'- y'|^2 + b - \frac{aK}{\ka} x_n,$$ with $K>0$ universal to be made precise later. By abuse of notation we continue to denote this function with $q$.

Next, we define a family of radial supersolutions as in Therem \ref{holder} (case A in Step 1):
$$\psi(x) = (1- Ka\eps ) u_0(d(x)),\quad x\in B_1,$$
with $$d(x): =\begin{cases}dist (x, \partial\cB),& x \in B_1 \cap \cB^c,\\ 
0, & x\in B_1\cap \cB,\end{cases}$$
where
$$\cB : = B_{\frac{1}{a\eps}}(-\frac{1}{a\eps} e_n).$$
By following the computations in \eqref{comp}, we get
\begin{align*}
    \Delta\psi\le(1-Ka\e)u_0''(d)(1+M\e a(n-1))\quad\text{in }\{\psi>0\}
\end{align*}
for some $M>0$ universal. It follows that
$$
\Delta\psi<(1-Ka\e)^{p-1}u_0^{p-1}(d)=\psi^{p-1}\quad\text{in }\{\psi>0\}
$$
as long as $(1-Ka\e)(1+M\e a(n-1))<1$ which holds if $K$ is large universal. We observe that by the homogeneity of $u_0,$
$$\widetilde{\psi} := \frac{u_0^{-1}(\psi) - x_n}{\eps} = \frac{1}{\eps} (d(1-a K \eps)^{\frac{1}{\ka}} - x_n) =\frac{1}{\eps} ( d - \frac{aK\eps}{\ka}d -x_n) + O(\eps).$$
Using that up to an $\eps^2$ error, $$d(x) = x_n + \frac{a\eps}{2}|x'|^2$$ we conclude that 
$$\widetilde{\psi} = \frac{a}{2}|x'|^2 - \frac{aK}{\ka} x_n + O(\eps) \quad \text{in $\overline{\{\psi>0\}}$}.$$ 
Then it is easily seen that for the translation
$$
\psi_\eps(x):=\psi(x-y' + b\eps e_n),
$$ 
it holds
\be\label{q}
\widetilde{\psi_\eps}:=\frac{u_0^{-1}(\psi_\e)-x_n}\e = q(x) + O(\eps) \quad \text{in $\overline{\{\psi_\eps>0\}}$}.
\ee 
We claim that for $k$ large, $\eps_k \to 0$, a translation $\psi_{\eps_k, \tau_k}$ of $\psi_{\eps_k}$ by  $O(\eps_k)$ in the $e_n$ direction touches $|\uu_k|$ by above at $x_k \to x_0$ with $x_k \in \Gamma(\psi_{\eps_k, \tau_k}) \cap \Gamma(\uu_k)$. This is a contradiction by the same argument as in Case A in Step 1 of the proof of Theorem \ref{holder}, where we considered a family of positive supersolutions.

To justify the claim, note that if $\psi_{\eps,\tau}$ represents the translation of $\psi_\eps$ by a vector $\tau\eps e_n,$
$$\psi_{\eps,\tau} = \psi_\eps(x+\tau \eps e_n),$$ then the graph of $\widetilde{\psi_{\eps,\tau}}=\frac{u_0^{-1}(\psi_{\e,\tau})-x_n}\e$ in $\R^{n+1}$ is translated by the vector $-\tau\e e_n+(\tau+O(\e))e_{n+1}$. Hence, by the uniform convergence of the graphs of $\widetilde{|\uu|}_k$ to $v$ and of $\widetilde{\psi_{\eps_k, \tau}}$ to $q+\tau $, by fixing $\tau=\tau_0>0$ small, 
we can conclude that, $\widetilde{\psi_{\eps_k, \tau_0}}$ is strictly larger than $\widetilde{|\uu|}_k$ in $\overline{\{|\uu_k|>0\}}\cap\overline{\{\psi_{\eps_k, \tau_0}>0\}}$  in a neighborhood of $x_0$.
This in turn gives that in a neighborhood $B_\eta(x_0),$
\be\label{ineq}\psi_{\eps_k, \tau_0} > |\uu_k| \quad \text{in {$B_\eta(x_0) \cap \overline{\{|\uu_k|>0\}} \cap\overline{\{\psi_{\eps_k, \tau_0}>0\}}.$}}\ee

On the other hand, from the definition of $\widetilde{|\uu|}_k$ it follows that $F(\uu_k)$ lies in a $o(\eps_k)$ neighborhood of the graph $x_n+\eps_k v(x',0)=0,$ 
while $F(\psi_{\eps_k, \tau_0})$ lies in a $o(\eps_k)$ neighborhood of $x_n = -\eps_k(q+\tau_0)$.  
Thus, $$\overline{\{|\uu_k| > 0\}} \subset \overline{\{\psi_{\eps_k,\tau_0}>0\}},$$ 
which means that \eqref{ineq} is satisfied everywhere in $ B_\eta(x_0) \cap \overline{\{ |\uu_k| > 0\}},$ and therefore 
\be\label{ineq1}\psi_{\eps_k, \tau_0} \geq |\uu_k| \quad \text{in {$B_\eta(x_0)$.}}\ee
Similarly, for $\tau=-\tau_0$ we use 
$$q-\tau_0 > v \quad \text{on $B_{2 \eta}(x_0) \setminus B_{\eta/4}(x_0)$}$$
and obtain 
\be\label{ineq11}\psi_{\eps_k, -\tau_0} \geq |\uu_k| \quad \text{in $B_\eta(x_0) \setminus B_{\eta/2}(x_0) $.}\ee
However, this inequality cannot hold in a full ball around $x_0$, otherwise it would imply the same inequality for $\widetilde{\psi_{\eps_k, -\tau_0}}$ and $\widetilde{|\uu|}_k$. This is clearly not true, as their graphs converge to $q-\tau_0$ and $v$, respectively. In conclusion, a translation $\psi_{\eps_k, \tau_k}$ with $\tau_k \in (-\tau_0,\tau_0),$ must touch $|\uu_k|$ by above at some point $x_k \in \overline{\{|\uu_k| >0\}} \cap B_{\eta/2}(x_0)$, $x_k \to x_0$, and since no touching can occur in the interior ($\psi$ is a strict supersolution), it must occur on the intersection of the free boundaries.

\medskip\noindent\emph{Step 2-2-2.}
Assume by contradiction that there exists a function 
$$q(x):=  -\frac 1 2 a |x'- y'|^2 + b + \gamma x_n^{1-s}, \quad \quad a,b \in \mathbb R, y' \in \mathbb {R}^{n-1},$$
with $\gamma>0, a\geq 1$, which touches $v$ by below at $x_0 \in B'_{1/2}.$
Set, $$\cB : = B_{\frac{1}{a\eps}}(\frac{1}{a\eps} e_n).$$ and let
$d(x)$ 
be the signed distance from $x$ to $\p \cB$ which is positive inside of $\cB$. 
Set
$$\Psi(t):= u_0(t) + \eps Ka (t+ \frac 12 t^2).$$ Then, following the computations in Case B in Step 1 of Theorem \ref{holder}, $\psi(x):=\Psi(d(x))$ is a strict subsolution in $B_1\cap \{\psi \neq 0\}$ for $K>0$ universal.

As before, we compute that for $t \in (0,1)$
\begin{align*}u_0^{-1}(\Psi(t)) &= t(1+ \eps \frac{Ka}{c_p}t^{1-\ka}(1+\frac 1 2 t))^{\frac{1}{\ka}}\\
&= t+ K_0 \eps a t^{2-\ka}(1 + \frac 1 2 t) + O(\eps^2 t^{3-2\ka}), \quad \text{where }K_0=\frac{K}{c_p\ka},\\
&= t+ K_0 \eps a t^{2-\ka}(1 + \frac 1 2 t) + O(\eps^2),\end{align*}
where we have used that $\ka \in (1,\frac 3 2 )$ when $s<1.$ Thus, 

$$\widetilde{\psi}(x):=\frac{u_0^{-1}(\psi)-x_n}{\e}= \frac{d-x_n}{\eps}+ K_0 a d^{2-\ka}(1 + \frac 1 2 d) + O(\eps)$$
and using that $d-x_n = -\frac{\eps a}{2} |x'|^2 + O(\eps^2)$ we get that 
$$\widetilde{\psi}(x)= -\frac{a}{2}|x'|^2+ K_0 a x_n^{2-\ka}(1 + \frac 1 2 x_n) + O(\eps) \quad \text{in $\overline{\{\psi>0\}}.$}$$
As before, after an appropriate translation of $\psi_\eps$ and up to appropriately modifying $q$ in a neighborhood of $x_0$,
$$\widetilde{\psi_\eps}= q + O(\eps) \quad \text{in $\overline{\{\psi_\eps>0\}}.$}$$ Precisely, we used that for $x_n$ small
$$\gamma x_n^{1-s} \geq K_0 a x_n^{2-\ka}(1 + \frac 1 2 x_n)$$
 as $1-s<2-\ka.$
As in Step 2-2-1, in a small ball $B_{\eta}(x_0)$, for $k$ large, $\eps_k \to 0$, a translation $\psi_{\eps_k,\tau_k}^+$ of $(\psi_{\eps_k})^+$ by  $o(\eps_k)$ in the $e_n$ direction touches $(u^1_k)^+$ by below at $x_k \to x_0$ with $x_k \in \{\psi_{\eps_k, \tau_k}=0\}$. We show that this translation touches $u_k^1$ by below at $x_k$, reaching a contradiction. For notational simplicity we call the translation $\psi_k$ 
and we still denote with $d$ the signed distance to $\{\psi_k=0\}$. We argue as in Case B in Step 1 of Theorem \ref{holder}: using that $(u^1_k)^-$ is subharmonic and $(u^1_k)^- \leq \eps_k^3$ in $B_{1/2}$, and $(u^1_k)^- \equiv 0 $ in $\{d \geq 0\}$ 
we conclude that in the region $B_{\eta/2}(x_0) \cap \{d<0\}$
$$(u^1_k)^- \leq  C(\eta)\eps_k^3|d|.$$
Hence, in the same region, for $k$ large
$$u^1_k \geq  C(\eta)\eps_k^3 d \geq \eps_k Ka (d+ \frac 12 d^2)=\psi_k, $$
as desired.

\medskip\noindent\emph{Step 3 (Convergence of the remaining components):} For $i=2,\ldots,m$, set  $v^i_k:=\eps_k^{-3/4}u^i_k$. Then, by assumption \eqref{34}, the $v^i_k$ are uniformly bounded and they satisfy
$$\Delta v^i_k = |\uu_k|^{p-2} v^i_k \quad \text{in $\{|\uu_k|>0\}.$}$$ 
By the flatness assumption \eqref{34}, interior H\"older estimates, and Lemma \ref{54}, we conclude that $v^i_k \to v^i$ uniformly on compact sets in $B^+_{1/2}$ with $v^i$ satisfying
\begin{equation}\label{components}
\begin{cases}
\Delta v^i = |u_0(x_n)|^{p-2}v^i &\text{in $B_{1/2}^+,$}\\
v^i =0 & \text{on $B_{1/2}\cap\{x_n\le0\}$}.
\end{cases}
\end{equation}
This problem can be reduced to \eqref{LiE} with $s=2\ka\in(2,4)$, after the transformation:
$$v^i = x_n^\ka w^i,$$ that is:
\begin{equation}
\Delta w^i + 2\ka\dfrac{ w^i_n}{x_n}=0  \quad\text{in $B_{1/2}^+,$}
\end{equation}
with $w^i$ uniformly bounded in $B_{1/4}^+$. The latter claim follows by standard arguments by comparing $v^i$ with a multiple of the explicit supersolution/subsolution $$\pm (C_1x_n^\ka + |x'-y'|^2 - C_2 x_n^2),\quad y'\in B'_{1/4},$$ with $C_1 \gg C_2$ universal. 

In particular, by the Lipschitz regularity of $w^i$ that follows from Theorem \ref{T7.2}, for $C>0$ universal,
\be\label{boh}|v^i - a_i u_0(x_n)| \leq C |x|^{1+\ka}\quad\text{in }B_{1/4}, \quad |a_i|\leq C.\ee
 
\medskip\noindent\emph{Step 4 (Improvement of flatness):}
Since $v$ satisfies \eqref{LiE} and $v(0)=0$, using the estimate in Theorem \ref{T7.2} and the convergence of the graphs of $\tilde u_k^1$ and $\widetilde{|\uu|}_k$ in the Hausdorff distance, we conclude that for $\rho$ small universal,
$$u_0(x_n + \eps_k \xi' \cdot x' - \eps_k \frac{\rho}{8}) \leq (u^1_k(x))^+ \leq |\uu_k| \leq u_0(x_n + \eps_k \xi'\cdot x' + \eps_k \frac{\rho}{8}), \quad |\xi'| \leq C,$$
holds in $B_\rho(\uu_k) \cup \Gamma(\uu_k)$. Then, thanks to the monotonicity of $u_0$, the inequality holds on the set where $|\uu_k|=0$,
hence it is valid in the whole $B_\rho.$ Denote
$$\nu:= \frac{e_n + \eps_k \xi'}{|e_n + \eps_k \xi'|} = e_n + \eps_k \xi' + O(\eps_k^2).$$ Then we easily deduce that in $B_\rho$,
\be\label{bbb} u_0(x \cdot \nu - \eps_k \frac{\rho}{4}) \leq (u^1_k(x))^+ \leq |\uu_k| \leq u_0(x \cdot \nu + \eps_k \frac{\rho}{4}).\ee
From \eqref{boh}, using that by the flatness assumption
$$u^1_k = u_0(x_n) + O(\eps_k),$$
 we conclude that for $c_0>0$ small universal to be chosen later, in $B_\rho$ ($\rho$ small universal)
\be\label{ua}|u^i_k - a_i \eps_k^{3/4}u^1_k| \leq c_0 \eps_k^{3/4}\rho^\ka,\quad i=2,\ldots,m.\ee
Set
$$\bar f:=\frac{f^1 + \eps_k^{3/4}\sum_{i \neq 1} a_i f^i}{\left|f^1 + \eps_k^{3/4}\sum_{i \neq 1} a_i f^i\right|} = f^1 + \eps_k^{3/4}\sum_{i \neq 1} a_i f^i + O(\eps_k^{3/2}).$$
Since $\bar f$ is orthogonal to $f^i - \eps_k^{3/4}a_i f^1$ for each $i=2,\ldots,m$, we conclude from \eqref{ua} that in $B_\rho$
$$|\uu_k - \mean{\uu_k, \bar f} \bar f | \leq (\frac{\eps_k}{2})^{3/4} \rho^\ka.$$
It remains to show that $\mean{\uu_k, \bar f}$ satisfies \eqref{non_dimpr}, hence reaching a contradiction.

From the definition of $\bar f$ we obtain that for $C_0>0$ universal
\begin{align}\label{eq:u-f}|\mean{\bar f,\uu_k}-u^1_k|\le C_0\e_k^{3/4}|u_k^i|\quad\text{in }B_\rho.\end{align}
We use this estimate, together with \eqref{bbb}
and
$$|u^i_k| \leq \min\{\eps^{3/4}, |\uu_k|\} \quad\text{in }B_1,\quad i=2,\ldots,m,$$ to prove the desired claim. Indeed, note that by \eqref{bbb}, we have in $\{\mean{x,\nu}\ge\e_k\frac\rho2\}$ $(u_k^1)^+>0$, thus $u_k^1=(u_k^1)^+$. If $t:=\mean{x, \nu}  -\frac{\eps_k}{4} \rho \in [\frac{\eps_k}{4} \rho, \eps_k^{1/3}]$ then
\begin{align*}
u_k^1-u_0(\mean{x,\nu}-\e_k\frac\rho2)&\ge u_0(t) - u_0(t-\frac{\eps_k}{4} \rho) \geq c \eps_k u_0'(t) \geq c \eps_k \frac{u_0(t)}{t}\\
&\geq C_0 \eps_k^{3/4} u_0(t+\eps_k \frac\rho 2)\ge C_0\e_k^{3/4}|\uu_k|\ge C_0\e_k^{3/4}|u_k^i|.
\end{align*}
Combining this with \eqref{eq:u-f} gives \eqref{non_dimpr}.
Similarly, if $t > \eps_k^{1/3}$, we use $\ka \in (1,2)$ to obtain
\begin{equation*}u_k^1-u_0(\mean{x,\nu}-\e_k\frac\rho2) \geq c \eps_k u_0'(t) \geq c \eps_k t \geq C_0 \eps_k^{3/2}.\qedhere\end{equation*}
\end{proof}


\section{Analyticity of the free boundary \\ 
(Proof of Theorem \ref{thm:main})}\label{sec:analy}




In this section, we prove that the flat free boundary for minimizers is analytic, which hinges on the $C^{1,\alpha}$-regularity results established in the previous sections, see Theorem \ref{main}. To achieve this, we adopt the approach from \cite{FotKoc24}, which relies on applying the partial Hodograph-Legendre transformation and the implicit function theorem.

We start with some technical results on the energy minimizer.

\begin{prop}[Weiss-type monotonicity formula]
    \label{prop:Weiss}
    Let $\uu$ be a minimizer to $E(\cdot)$ in $B_1$. For $x_0\in B_{1/2}$ and $r\in(0,1/2)$, we set
    $$
    W(\uu,x_0,r):=\frac1{r^{n+2\ka-2}}\left[\int_{B_r(x_0)}\left(|\D \uu|^2+\frac2p|\uu|^p\right)-\frac\ka{r}\int_{\partial B_r(x_0)}|\uu|^2\right].
    $$
    Then
    \begin{enumerate}
        \item $W(\uu,x_0,r)$ is nondecreasing in $r$.
        \item If $W(\uu,x_0,r)$ is constant for $0<r<r_0$, then $\uu$ is homogeneous of degree $\ka$ with respect to $x_0$ in $B_{r_0}(x_0)$.
        \item $x\longmapsto W(\uu,x,0+)$ is upper-semicontinuous.
    \end{enumerate}
\end{prop}

\begin{proof}
By arguing as in the proof of \cite[Proposition~2.3]{FotShaWei21}, we can obtain that for $0<t<s<r_0$,
\begin{align*}
    W(\uu,x_0,s)-W(\uu,x_0,t)=\int_t^s\int_{\partial B_1}\frac2r|\mean{\D \uu_r,x}-\ka \uu_r|^2dS_xdr,
\end{align*}
where $\uu_r(x)=\frac{\uu(x_0+rx)}{r^\ka}$. This implies the first and the second statements in Proposition~\ref{prop:Weiss}. The last one follows from the first one.    
\end{proof}

We consider the class  of half-space solutions of \eqref{eq:system}:
\begin{align}\label{eq:half-space-sol}
    \Hf:=\{x\longmapsto c_p \mean{x, e}_+^\ka f\,:\, \text{$e\in\R^n$ and $f\in\R^m$ are unit vectors}\}.
\end{align}

\begin{lem}\label{lem:Half-space}
    Let $\uu\not\equiv0$ be a homogeneous global solution of \eqref{eq:system}. Suppose that $\uu=0$ in $\{x_n<0\}$ and $W(\uu,0,1)=\omega_p$, where $\omega_p$ is the Weiss energy of half-space solutions in $\Hf$. Then $\uu\in\Hf$.
\end{lem}

\begin{proof}
By the homogeneity of $\uu$, $\cC:=\{|\uu|>0\}$ is a cone. We split the proof into three steps.

\medskip\noindent\emph{Step 1.} In this step, we show that for any point $x_0\in\partial\cC\setminus\{0\}$, $\uu$ is homogeneous of degree $\ka$ with respect to $x_0$. Indeed, for any $j>0$, we have by homogeneity of $\uu$
\begin{align*}
    W(\uu,x_0,0+)=\lim_{r\to0+}W(\uu,x_0,r)=\lim_{r\to0+}W(\uu,x_0/j,r/j)=W(\uu,x_0/j,0+).
\end{align*}
By the upper semicontinuity of $x\longmapsto W(\uu,x,0+)$, we further have
\begin{align}
    \label{eq:weiss-x_0}
    W(\uu,x_0,0+)=\limsup_{j\to\infty}W(\uu,x_0/j,0+)\le W(\uu,0,0+)=\omega_p.
\end{align}
By following the first part of the proof in \cite[Proposition~4.6]{FotShaWei21}, we can obtain that if $\vv\not\equiv0$ is a homogeneous solution of degree $\ka$ satisfying $\{|\vv|=0\}^\circ\neq\emptyset$, then $W(\vv,0,1)\ge\omega_p$. Then \eqref{eq:weiss-x_0} gives $W(\uu,x_0,0+)=\omega_p$. Moreover, $W(\uu,x_0,\infty)=W(\uu,0,\infty)=\omega_p$. Thus, $\uu$ is homogeneous of degree $\ka$ with respect to $x_0$.

\medskip\noindent\emph{Step 2.} In this step, we show that $\cC=\{x_n>0\}$. We assume to the contrary $\cC\neq\{x_n>0\}$. Then, since $\cC\subset\{x_n>0\}$, we can find two distinct points $y_1,y_2\subset\partial\cC\setminus\{0\}$ such that $[y_1,y_2]\cap\cC=(y_1,y_2)$. Since $y_1\in\partial\cC\setminus\{0\}$, the result in Step 1 implies that for all $x\in[y_1,y_2]$, $|\uu(x)|=C|x-y_1|^\ka$ for some $C>0$. This contradicts $\uu(y_2)=0$.

\medskip\noindent\emph{Step 3.} The aim of this step is to show that $\uu\in\Hf$. Indeed, from results in the previous steps, $\{|\uu|>0\}=\{x_n>0\}$ and $\uu$ is homogeneous with respect to any point on $\{x_n=0\}$. By the homogeneity of $\uu$ with respect to $0$, $\uu(0,x_n)=u(0,1)x_n^\ka$ for any $x_n>0$. Then, whenever $x'\in\R^{n-1}\setminus\{0\}$ and $x_n>0$, the homogeneity of $\uu$ with respect to $\left(x'+x_n\frac{x'}{|x'|},0\right)$ gives 
$$
\uu(x',x_n)=\uu(0,x_n+|x'|)\left(\frac{x_n}{x_n+|x'|}\right)^\ka=\uu(0,1)x_n^\ka.
$$
This completes the proof.    
\end{proof}

\begin{lem}
    \label{lem:unique-blowup}
    Let $\uu$ be an energy minimizer to $E(\cdot)$ in $B_1$. Suppose that $\uu$ is $\bar\e$-flat in $B_1$ for some $\bar\e>0$ small universal, and $0\in\Gamma(\uu)$. Then there are unique unit vectors $e_0\in\R^n$ and $f_0\in \R^m$ such that as $r\to0$
    $$
    \frac{\uu(rx)}{r^\ka}\to c_p\mean{x,e_0}_+^\ka f_0,\quad x\in B_1.
    $$
\end{lem}

\begin{proof}
    Once we have the improvement of flatness lemma (Lemma~\ref{IMPF}), Lemma~\ref{lem:unique-blowup} follows by a standard technique, see e.g., \cite[Lemma~8.4]{Vel23}.
\end{proof}

\begin{lem}\label{lem:blowup-diff-center}
    Let $\uu$ be an energy minimizer to $E(\cdot)$ in $B_1$. Suppose that $\uu$ is $\bar\e$-flat in $B_1$ for some $\bar\e>0$ small universal, and $0\in\Gamma(\uu)$. If $\{|\uu|>0\}\cap B_1\ni x_i\to 0$, then for $d_i:=\dist(x_i,\Gamma(\uu))$, 
    $$
    \uu_i(x):=\frac{\uu(x_i+d_ix)}{d_i^\ka}\to c_p\mean{x+e_0,e_0}_+^\ka f_0
    $$
    in $C^{1,\gamma}(B_{1/2})$ for any $0<\gamma<1$, where $e_0$ and $f_0$ are as in Lemma~\ref{lem:unique-blowup}.
\end{lem}

\begin{proof}
The proof of this follows that of \cite[Lemma~2.4]{FotKoc24}. However, we give the full proof since there are technical differences. We divide the proof into three steps.

\medskip\noindent\emph{Step 1.} For each $i$, $y_i\in \partial B_{d_i}(x_i)\cap \Gamma(\uu)$. Since $\Gamma(\uu)$ is $C^{1,\al}$
 by Theorem~\ref{main}, we have by Lemma~\ref{lem:unique-blowup} that $\nu_i:=\frac{x_i-y_i}{d_i}\to e_0$ as $i\to\infty$. Moreover, thanks to \eqref{flat_imp}, we can find universal constants $C>0$ and $r_0>0$ such that $\sup_{B_r}|\uu|\le Cr^\ka$, $0<r<r_0$. This, along with the Weiss-type monotonicity formula (Proposition~\ref{prop:Weiss}), gives
 $$
 E(\uu,r)\le r^{n+2\ka-2}E(\uu,1)+\frac\ka r\int_{\partial B_r}|\uu|^2\le Cr^{n+2\ka-2}(E(\uu,1)+1).
 $$
Since $\uu$ is $\bar\e$-flat in $B_1$, we have that for large $i$,  $\uu(\cdot+y_i)$ is $2\bar\e$-flat in $B_{1/2}$, hence
$$
E(\uu,y_i,r)\le Cr^{n+2\ka-2}(E(\uu,1)+1),\quad 0<r<r_0/2.
$$
This implies that for large $i$,
$$
E(\uu_i,1)\le C(E(\uu,1)+1).
$$
Since $\uu_i$ is a minimizer of $E(\cdot)$ in $B_1$, we have by Theorem~\ref{thm:grad-u-holder} that $\uu_i\in C^{1,\ka-1}(\overline{B_{1/2}})$ with
$$
\|\uu_i\|_{C^{1,\ka-1}(\overline{B_{1/2}})}\le C(E(\uu,1)+1).
$$
Thus, for any $0<\gamma<1$, $\uu_i\to\uu_0$ in $C^{1,\gamma}(\overline{B_{1/2}})$ for some global solution $\uu_0$. 

\medskip\noindent\emph{Step 2.} For $x\in\Gamma(\uu)$ near $0$, $\lim_{r\to0+}W(\uu,x,r)=\omega_p$ by Lemma~\ref{lem:unique-blowup}. Since $r\longmapsto W(\uu,x,r)$ is monotone, the convergence is uniform by Dini's monotone convergence theorem. Then, for $0<r<1$,
\begin{align*}
    W(\uu,-e_0,r)=\lim_{i\to\infty}W(\uu_i,-\nu_i,r)=\lim_{i\to\infty}W(\uu,y_i,rd_i)=W(\uu,0,0+)=\omega_p.
\end{align*}
Thus, $\uu_0$ is homogeneous of degree $\ka$ with respect to $-e_0$ by Proposition~\ref{prop:Weiss}.

Moreover, we observe that for any $\e>0$, there is $\de>0$ such that for large $i$
$$
\left\{\lmean{\frac{x-y_i}{|x-y_i|}, \nu_i}<-\e \right\}\cap B_\de\subset \{\uu=0\},
$$
which, by a direct computation, is equivalent to
$$
\left\{\lmean{\frac{x+\nu_i}{|x+\nu_i|}, \nu_i}<-\e \right\}\cap B_{\de/d_i}\left(-x_i/d_i\right)\subset \{\uu_i=0\},
$$
hence
$$
\left\{\lmean{\frac{x+e_0}{|x+e_0|}, e_0}<-\e \right\}\cap B_\de\subset \{\uu_0=0\}.
$$
Since $\e$ is arbitrary, 
$$
\{\mean{x+e_0,e_0}<0\}\subset\{\uu_0=0\}.
$$

By summing up the above results, we have by Lemma~\ref{lem:Half-space} that
$$
\uu_0(x)=c_p\mean{x+e_0,e_0}_+^\ka f
$$
for some unit vector $f\in\R^m$.

\medskip\noindent\emph{Step 3.} It remains to show $f=f_0$. To this end, we observe
\begin{align*}
    c_pf=\uu_0(0)=\lim_{i\to\infty}\uu_i(0)=\lim_{i\to\infty}\frac{\uu(x_i)}{d_i^\ka}=\lim_{i\to\infty}\left(\frac{|x_i|}{d_i}\right)^\ka\frac{\uu(x_i)}{|x_i|^\ka}.
\end{align*}
Over a subsequence, $\frac{x_i}{|x_i|}\to x_*$ for some $x_*\in\partial B_1$, thus
$$
\lim_{i\to\infty}\frac{\uu(x_i)}{|x_i|^\ka}=\lim_{i\to\infty}\frac{\uu\left(|x_i|\frac{x_i}{|x_i|}\right)}{|x_i|^\ka}=c_p\mean{x_*,e_0}_+^\ka f_0.
$$
Therefore, $f=f_0$, as desired.
\end{proof}

We are now ready to prove analyticity of the free boundary for minimizers.



\medskip
\noindent
{\it Proof of Theorem \ref{thm:main}:}

In \cite{FotKoc24}, the authors obtained a similar result when $1\le p<2$ by using the following arguments: they used the partial Legendre transformation to reduce the regularity problem of the free boundary to a system which can be considered as a perturbation of a linear degenerate operator
$$
x_n\Delta+\gamma\partial_n,\qquad\text{where }\gamma=2(\ka-1).
$$
Then they applied the implicit function theorem to conclude the analyticity of the solution.

In our case $0<p<1$, once we have the optimal regularity of solutions of \eqref{eq:system}, the $C^{1,\alpha}$-regularity of the flat free boundary, and Lemma~\ref{lem:blowup-diff-center}, we can follow the argument above in \cite{FotKoc24} to prove that the flat free boundary is analytic. In fact, one can easily see that the argument in \cite[sections 3-5]{FotKoc24} regarding the partial hodograph-Legendre transformation holds as long as $\ka>1$.
\qed 


 \section*{Acknowledgements}
S. J. was supported by the Academy of Finland grant 347550 and the research fund of Hanyang University(HY-202400000003278). H.S. was supported by the Swedish Research Council (grant no.~2021-03700).





\end{document}